\documentclass[11pt,reqno]{amsart}
\usepackage{amsmath,amssymb,graphicx,mathrsfs,amsopn,stmaryrd,amsthm,color,bm,times}
\usepackage[colorlinks=true,citecolor=blue,linkcolor=blue,pagebackref]{hyperref}
\usepackage[final]{showlabels}
\usepackage[english]{babel}
\usepackage[mathscr]{euscript}
\usepackage{xcolor}
\usepackage{geometry}
\let\emptyset\varnothing
\usepackage{tikz}    
\usetikzlibrary{arrows,matrix,decorations.markings,shapes.geometric}   
\usetikzlibrary{decorations.pathreplacing}

\numberwithin{equation}{section}
\newtheorem{thm}{Theorem}[section]
\newtheorem{prop}[thm]{Proposition}
\newtheorem{lem}[thm]{Lemma}
\newtheorem{cor}[thm]{Corollary}
\newtheorem{conj}[thm]{Conjecture}
\newtheorem*{opq}{Open Problem}
\theoremstyle{definition} 
\newtheorem{eg}[thm]{Example}

\theoremstyle{remark}
\newtheorem{rem}[thm]{Remark}

\newcommand{\beq}{\begin{equation}}
\newcommand{\eeq}{\end{equation}}
\newcommand{\be}{\begin{equation*}}
\newcommand{\ee}{\end{equation*}}
\newcommand{\bs}{\boldsymbol}

\newcommand{\C}{\mathbb{C}}

\newcommand{\Z}{\mathbb{Z}}

\newcommand{\mc}{\mathcal}

\newcommand{\g}{\mathfrak{g}}
\newcommand{\gl}{\mathfrak{gl}}
\newcommand{\h}{\mathfrak{h}}

\newcommand{\End}{\mathrm{End}}

\newcommand{\pa}{\partial}
\newcommand{\tl}{\tilde}
\newcommand{\gge}{\geqslant}
\newcommand{\lle}{\leqslant}
\newcommand{\la}{\lambda}
\newcommand{\La}{\Lambda}

\newcommand{\glMN}{\mathfrak{gl}_{m|n}}
\newcommand{\UglMN}{\mathrm{U}(\mathfrak{gl}_{m|n})}
\newcommand{\UqglMN}{\mathrm{U}_q(\mathfrak{gl}_{m|n})}
\newcommand{\UqglMNh}{\mathrm{U}_q(\widehat{\mathfrak{gl}}_{m|n})}

\newcommand{\YglMN}{\mathrm{Y}(\mathfrak{gl}_{m|n})}
\newcommand{\YglN}{\mathrm{Y}(\mathfrak{gl}_{N})}

\newcommand{\Yone}{\mathrm{Y}(\mathfrak{gl}_{1|1})}
\newcommand{\bmx}{\begin{pmatrix}}    
\newcommand{\emx}{\end{pmatrix}}   
\newcommand{\wt}{\widetilde}

\newcommand{\qedd}{\tag*{$\square$}}

\begin{document}
\pagestyle{myheadings}
\setcounter{page}{1}

\title[Gelfand-Tsetlin Bases]{Gelfand-Tsetlin bases of representations for super Yangian\\ and quantum affine superalgebra}

\author{Kang Lu}
\address{K.L.: Department of Mathematics, University of Denver, 
\newline
\strut\kern\parindent 2390 S. York St., Denver, CO 80208, USA}
\email{kang.lu@du.edu}

\maketitle

\begin{abstract} 
		We give explicit actions of Drinfeld generators on Gelfand-Tsetlin bases of super Yangian modules associated with skew Young diagrams. In particular, we give another proof that these representations are irreducible. We study irreducible tame $\Yone$-modules and show that a finite-dimensional irreducible $\Yone$-module is tame if and only if it is thin. We also give the analogous statements for quantum affine superalgebra of type A. 
		\medskip 
		
\end{abstract}


\thispagestyle{empty}
\section{Introduction}	
Yangians and quantum affine algebras and their representations have been extensively studied since 1980. Many striking results are produced. Though super Yangian $\YglMN$ of general linear Lie superalgebra $\glMN$ was introduced in \cite{Naz} and its finite-dimensional irreducible modules were classified in \cite{Z95,Z96}, only a few works were done for $\YglMN$ and representations of $\YglMN$ are still far from being well understood. Continuing \cite{LM20}, we study further skew representations of super Yangian which were introduced in \cite{Che87} and intensively studied in \cite{NT98,NT98b,NT02,Naz04} for the even case. 

Inside of the super Yangian, there is a distinguished maximal commutative subalgebra $\mathrm A(\glMN)$ generated by the Cartan currents of $\YglMN$ which we call the {\it Gelfand-Tsetlin algebra}. We say that a finite-dimensional $\YglMN$-module $M$ is  {\it tame} if the action of the subalgebra $\mathrm A(\glMN)$ on $M$ is semi-simple. We call $M$ {\it thin} if $M$ is tame and the spectrum of $\mathrm A(\glMN)$ on $M$ is simple.

Skew representations are a certain family of finite-dimensional $\YglMN$-modules including evaluation covariant (polynomial) modules.  They have bases parameterized by Gelfand-Tsetlin patterns (or semi-standard Young tableaux of the associated skew Young diagrams) and hence are called {\it Gelfand-Tsetlin bases}. It turns out that these bases are indeed eigenbases of the Gelfand-Tsetlin algebra. Therefore, skew representations are tame. Moreover, the eigenvalues can be computed explicitly and it is not hard to see that skew representations are actually thin. According to \cite[Proposition 3.1]{You15}, the action of the non-Cartan currents of Drinfeld generators on an eigenvector of $\mathrm A(\glMN)$ in a thin module is essentially determined by the action of the first coefficients of the non-Cartan currents. Combining with \cite[Theorem 7]{SV10}, we give the matrix elements of each currents acting on skew representations with respect to Gelfand-Tsetlin bases. In particular, it describes explicitly the poles of the currents acting on skew representations. It would be interesting to determine the set of poles of the currents acting on an arbitrary finite-dimensional irreducible module, cf. \cite{GW20}. It is also interesting to generalize \cite{NT94,Mol94} to the super setting where the main obstacle is the absence of polynomial action of Drinfeld type currents.

As a corollary, we show that skew representations of $\YglMN$ are irreducible. Note that the irreducibility of skew representations is obtained in \cite[Theorem 4.9]{LM20} using the general fact that the Drinfeld functor maps an  finite-dimensional irreducible module of degenerate affine Hecke algebra to a finite-dimensional irreducible module of super Yangian, see \cite[Proposition 4.8]{LM20}. Here we provide another independent proof of the irreducibility of skew representations. Note that the irreducibility should also follow from the super analogue of the centralizer construction in \cite{MO00}. 

\medskip

The result of this paper is a step towards understanding tame modules of $\YglMN$. Let $t_{ij}(u)$ be the R-matrix presentation generating series of the super Yangian $\YglMN$, where $t_{ij}(u)$ are series in $u^{-1}$ with $\delta_{ij}$ as the constant term  and certain generators of $\YglMN$ as other coefficients, see Section \ref{sec rtt}. Given $\xi(u)\in 1+u^{-1}\C[[u^{-1}]]$, let $\C_{\xi(u)}$ be the one-dimensional module spanned by a nonzero vector $v$ satisfying $t_{ii}(u)v=\xi(u)v$ and $t_{ij}(u)v=0$ for $1\lle i\ne j\lle m+n$. For any $z\in \C$ and any skew Young diagram $\la/\mu$, denote by $L_z(\la/\mu)$ the skew representation corresponding to the skew Young diagram $\la/\mu$ with evaluation parameter $z$. It was conjectured in \cite{Che87} and classified in \cite{NT98} that all finite-dimensional irreducible tame modules of the (nonsuper) Yangian $\YglN$ are, up to isomorphism, of the form
\beq\label{eq:tensor-skew}
\C_{\xi(u)}\otimes L_{z_1}(\la_1/\mu_1)\otimes \cdots L_{z_k}(\la_k/\mu_k),
\eeq
where $\xi(u)\in 1+u^{-1}\C[[u^{-1}]]$, $k\in \Z_{\gge 0}$, and $z_i-z_j\notin \Z$ for all $1\lle i< j \lle k$. Here $\la_i/\mu_i$ is a skew Young diagram for each $1\lle i\lle k$. Hence skew representations are the elementary but also fundamental objects among tame modules in the even case which motivates the study in the supersymmetric setting.

We take the opportunity to list a few open problems about irreducible tame $\YglMN$-modules. First, it would be interesting to generalize the classification of irreducible tame modules to super case.

\begin{opq}
Classify all finite-dimensional irreducible tame $\YglMN$-modules.\qed
\end{opq}

One might suggest again that up to a one-dimensional module, finite-dimensional irreducible tame modules are given by tensor products of skew representations with evaluation parameters in distinct $\Z$-cosets. Note that skew representations are direct sums of covariant representations of $\glMN$ when restricted as $\glMN$-modules. It is not hard to see that $\YglMN$-modules of the form in \eqref{eq:tensor-skew} are thin and hence are tame. However, they do not cover all finite-dimensional irreducible tame $\YglMN$-modules as there are finite-dimensional irreducible tame modules that are not sub-quotients of tensor powers of evaluation vector representations. The simplest examples are 2-dimensional evaluation $\Yone$-modules of non-integral weights. Even if we consider only the finite-dimensional irreducible tames modules whose restrictions are direct sums of covariant representations of $\glMN$, there are still such finite-dimensional irreducible tame modules that are not of the form \eqref{eq:tensor-skew}. An example of such case will be given in Section \ref{sec:Yone}.

It was shown for $\YglN$ in \cite{NT98} and for quantum affine algebras of type B in \cite{BM17} that a finite-dimensional irreducible module is tame if and only if it is thin. We believe the same statement also holds for $\YglMN$. 

\begin{conj}\label{conj:main}
A finite-dimensional irreducible $\YglMN$-module is tame only if it is thin.\qed
\end{conj}

One of our main results is to prove Conjecture \ref{conj:main} for the case $m=n=1$ in Section \ref{sec:Yone}.

\medskip

The paper is organized as follows. We recall the Gelfand-Tsetlin bases for covariant representations of $\glMN$ and prepare basic facts of super Yangian in Section \ref{sec:super yangian}. In Section \ref{sec main}, we give our main results for super Yangian and their proofs. We give analogous results for quantum affine superalgebra $\UqglMNh$ with generic $q$ in Section \ref{sec q}.

\medskip

{\bf Acknowledgments.} The author thanks E. Mukhin and V. Tarasov for stimulating discussions.

\section{Super Yangian $\YglMN$}\label{sec:super yangian}
\subsection{Lie superalgebra $\glMN$}\label{sec glmn}
Throughout the paper, we work over $\C$. A \emph{vector superspace} $W = W_{\bar 0}\oplus W_{\bar 1}$ is a $\Z_2$-graded vector space. We call elements of $W_{\bar 0}$ \emph{even} and elements of
$W_{\bar 1}$ \emph{odd}. We write $|w|\in\{\bar 0,\bar 1\}$ for the parity of a homogeneous element $w\in W$. Set $(-1)^{\bar 0}=1$ and $(-1)^{\bar 1}=-1$.

Fix $m,n\in \Z_{\gge 0}$. Set $I:=\{1,2,\dots,m+n-1\}$ and $\bar I:=\{1,2,\dots,m+n\}$. We also set $|i|=\bar 0$ for $1\lle i\lle m$ and $|i|=\bar 1$ for $m< i\lle m+n$. Define $s_i=(-1)^{|i|}$ for $i\in \bar I$.

The Lie superalgebra $\glMN$ is generated by elements $e_{ij}$, $i,j\in \bar I$, with the supercommutator relations
\[
[e_{ij},e_{kl}]=\delta_{jk}e_{il}-(-1)^{(|i|+|j|)(|k|+|l|)}\delta_{il}e_{kj},
\]
where the parity of $e_{ij}$ is $|i|+|j|$. Set $e_{i}:=e_{i,i+1}$ and $f_i:=e_{i+1,i}$ for $i\in I$. Denote by $\UglMN$ the universal enveloping superalgebra of $\glMN$.

The \emph{Cartan subalgebra $\h$} of $\glMN$ is spanned by $e_{ii}$, $i\in \bar I$. Let $\epsilon_i$, $i\in I$, be a basis of $\h^*$ (the dual space of $\h$) such that $\epsilon_i(e_{jj})=\delta_{ij}$. There is a bilinear form $(\ ,\ )$ on $\h^*$ given by $(\epsilon_i,\epsilon_j)=s_i\delta_{ij}$. Define the simple roots $\alpha_i:=\epsilon_i-\epsilon_{i+1}$, for $i\in I$.

Let $\la=(\la_1,\la_2,\dots,\la_{m+n})$ be a tuple of complex numbers. We call $\la$ a $\glMN$-weight. Denote $L(\la)$ the irreducible module of $\glMN$ generated by a nonzero vector $v$ satisfying the conditions 
\[
e_{ii}v=\la_i v,\quad e_{jk}v=0,
\]
for $i\in \bar I$ and $1\lle j< k\lle m+n$. 

Let $\mathcal V:=\C^{m|n}$ be the vector superspace with a basis $v_i$, $i\in \bar I$, such that $|v_i|=|i|$. Let $E_{ij}\in\End(\mathcal V)$ be the linear operators such that $E_{ij}v_k=\delta_{jk}v_i$. The map $\rho_{\mathcal V}:\glMN\to \End(\mathcal V),\ e_{ij}\mapsto E_{ij}$ defines a $\glMN$-module structure on $\mathcal V$. We call it the \emph{vector representation} of $\glMN$. The highest weight of $\mathcal V$ is the tuple $(1,0,\dots,0)$. 

We call $\la$ a \emph{covariant $\gl_{m|n}$-weight} if $\la$ satisfies: all $\la_1,\dots,\la_{m+n}$ are nonnegative integers and the number $l$ of nonzeo components among $\la_{m+1},\dots,\la_{m+n}$ does not exceed $\la_m$, see \cite{Ser85,BR87}; moreover, $\la_1\gge \dots\gge \la_{m}$ and $\la_{m+1}\gge \dots\gge \la_{m+n}$. 

We call $L(\la)$ a \emph{covariant module} if $\la$ is a covariant $\gl_{m|n}$-weight. Note that in this case $L(\la)$ is a submodule of $\mathcal V^{\otimes |\la|}$, where $|\la|=\sum_{i=1}^{m+n}\la_i$.

Fix $r\gge 0$. For $k\in \Z_{\gge 0}$, we set $k'=r+k$. 

Define similar notations for the Lie algebra $\gl_{r}:=\gl_{r|0}$ and the Lie superalgebra $\gl_{m'|n}$.

\subsection{Gelfand-Tsetlin tableaux}
We identify $\gl_{r}$ as a Lie subalgebra of $\gl_{m'|n}$ via the natural embeding $e_{ij}\mapsto e_{ij}$ and $\gl_{m|n}$ as a Lie subalgebra of $\gl_{m'|n}$ via the embeding $e_{ij}\mapsto e_{i'j'}$. It is clear that $\gl_{r}$ commutes with $\gl_{m|n}$. 

Let $\la=(\la_1,\dots,\la_{m'+n})$ be a covariant $\gl_{m'|n}$-weight and $\mu=(\mu_1,\dots,\mu_{r})$ a covariant $\gl_r$-weight. Let $L(\la)$ be the corresponding irreducible $\gl_{m'|n}$-module. Regard $L(\la)$ as a $\gl_r$-module. Let $L(\la/\mu)$ be the subspace of $L(\la)$ given by
\[
L(\la/\mu):=\{v\in L(\la)~|~e_{ii}v=\mu_i v, e_{jk}v=0, \text{ for }1\lle i\lle r, 1\lle j<k \lle r\}.
\]
Clearly, $L(\la/\mu)$ is a $\gl_{m|n}$-module and a $\mathrm U(\gl_{m'|n})^{\gl_r}$-module. 

Our main combinatorial device is an array of complex numbers 
 $\Lambda=(\lambda_{ij})$ presented in the following form:
 \begin{equation}\label{GT pattern}
\begin{array}{cccccccc}
   \la_{m'+n,1}   &  \cdots      & \la_{m'+n,m'} & \la_{m'+n,m'+1} & \cdots & \la_{m'+n,m'+n-1} & \la_{m'+n,m'+n} \\
  \la_{m'+n-1,1} &  \cdots     & \la_{m'+n-1,m'} & \la_{m'+n-1,m'+1} & \cdots & \la_{m'+n-1,m'+n-1} & \\
  \vdots    &  \vdots & \vdots & \vdots   & \reflectbox{$\ddots$}    \\
  \la_{r+1,1} &  \cdots  & \la_{r+1,r} & \la_{r+1,r+1}    \\
  \la_{r1} &  \cdots  & \la_{rr}                           
\end{array}
\end{equation}
 We call  $\Lambda$ a  \emph{Gelfand-Tsetlin tableau} (GT tableau for short).  Given $\Lambda=(\lambda_{ij})$, we set
\begin{equation}\label{eq:lki}
l_{ki}=\la_{k'i}+r-i+1, (1\lle i\lle m'); \quad l_{kj}=-\la_{k'j}+r+j-2m', (m'+1\lle j\lle k').
\end{equation}
A GT tableau is \emph{$\la/\mu$-admissible} if the following conditions are satisfied:
\begin{enumerate}
\item $\la_{m'+n,i}=\la_i$ and $\la_{rj}=\mu_j$ for $1\lle i\lle m'+n$ and $1\lle j\lle r$;
\item $\theta_{k-1,i}:=\la_{ki}-\la_{k-1,i}\in \{0,1\}$, $1\lle i\lle m'$, $m'+1\lle k\lle m'+n$;
\item $\la_{km'}\gge \#\{i: \la_{ki}>0, m' \lle i\lle k\}$, $m' +1\lle k\lle m'+n$;
\item if $\la_{m'+1,m'}=0$, then $\theta_{m'm'}=0$;
\item $\la_{ki}-\la_{k,i+1}\in \Z_{\gge 0}$, $1\lle i\lle m'-1$, $m'+1\lle k\lle m'+n$;
\item $\la_{k+1,i}-\la_{ki}\in\Z_{\gge 0}$ and $\la_{ki}-\la_{k+1,i+1}\in\Z_{\gge 0}$, $1\lle i\lle k\lle m'-1$ or $m'+1\lle i\lle k\lle m'+n-1$.
\end{enumerate}

We recall the following theorem from \cite{SV10}. Here we adopt the renormalized version from \cite[Theorem 6.1]{FSZ20}. Note that our $\La$ corresponds to $\La$ in \cite{FSZ20} with $\la_{ki}=\mu_i$ for $1\lle i\lle k\lle r$ as we consider the subspace of singular vectors of $\gl_r$-weight $\mu$ in $L(\la)$. 

\begin{thm}[{\cite[Theorem 7]{SV10}}]\label{thm:gt-formula}
The $\gl_{m|n}$-module $L(\la/\mu)$ admits a basis $\xi_{\La}$ parameterized by all $\la/\mu$-admissible GT tableaux $\La$. The actions of the generators of $\gl_{m|n}$
are given by the formulas
\begin{equation*}
e_{kk}\xi_{\La}=\Big(\sum_{i=1}^{k'}\la_{k'j}-\sum_{j=1}^{k'-1}\la_{k'-1,j}\Big)\xi_{\La},\quad 1\lle k\lle m+n;
\end{equation*}

\begin{equation*}
\noindent e_{k}\xi_{\La}=-\sum_{i=1}^{k'}\frac{\Pi_{j=1}^{k'+1}(l_{k+1,j}-l_{ki}) }
  {\Pi_{j\neq i,j=1}^{k'} (l_{kj}-l_{ki}) }\xi_{\La+\delta_{ki}},
\quad 1\lle k\lle m-1;
\end{equation*}

\begin{equation*}
f_{k}\xi_{\La}=\sum_{i=1}^{k'}\frac{\Pi_{j=1}^{k'-1}(l_{k-1,j}-l_{ki})}
{\Pi_{j\neq i,j=1}^{k'}(l_{kj}-l_{ki})}\xi_{\La-\delta_{ki}},
\quad  1\lle k\lle m-1;
\end{equation*}

\begin{equation*}
\begin{split}
e_{m}\xi_{\La}=&\sum_{i=1}^{m'}\theta_{m' i}(-1)^{i-1}(-1)^{\theta_{m' 1}+\ldots+\theta_{m',i-1}}\\& \times  \frac{\Pi_{1\lle j< i\lle m'} (l_{m j}-l_{m i}-1)}
  {\Pi_{1\lle i<j\lle m'} (l_{m j}-l_{m i})
    \Pi_{j\neq i,j=1}^{m'}(l_{m+1,j}-l_{m i}-1)}
    \xi_{\La+\delta_{m i}},
\end{split}
\end{equation*}

\begin{equation*}
\begin{split}
f_{m} \xi_{\La}=&\sum_{i=1}^{m'}(1-\theta_{m' i})(-1)^{i-1}(-1)^{\theta_{m' 1}+\ldots+\theta_{m',i-1}}
\\ & \times  \frac{(l_{mi}-l_{m+1,m'+1})\Pi_{ 1\lle i<j\lle m'} (l_{m j}-l_{m i}+1)\Pi_{j=1}^{m'-1}(l_{m-1,j}-l_{m i})}
  {\Pi_{1\lle j< i\lle m'} (l_{mj}-l_{mi})} \xi_{\La-\delta_{m i}},
\end{split}
\end{equation*}

and for $m+1\lle k\lle m+n-1$,

\begin{equation*}
\begin{split}
e_{k} \xi_{\La}=&\sum_{i=1}^{m'}\theta_{k'i}(-1)^{\vartheta_{k'i}}(1-\theta_{k'-1,i})\times
\prod_{j\neq i,j =1}^{m'}\left(\frac{l_{kj}-l_{ki}-1}{l_{k+1,j}-l_{ki}-1}\right)
\xi_{\La+\delta_{ki}}
\\
&-\sum_{i=m'+1}^{k'} 
\prod_{j=1}^{m'}\left(\frac{(l_{kj}-l_{ki})(l_{kj}-l_{ki}+1)}{(l_{k+1,j}-l_{ki})(l_{k-1,j}-l_{ki}+1)} \right)
\times
  \frac{\Pi_{j=m'+1}^{k'+1}(l_{k+1,j}-l_{ki})}
  {\Pi_{j\neq i,j=m'+1}^{k'} (l_{kj}-l_{ki})}\xi_{\La+\delta_{ki}} ,
\end{split}
\end{equation*}

\begin{equation*}
\begin{split}
f_{k}\xi_{\La}=&
\sum_{i=1}^{m'}\theta_{k'-1,i}(-1)^{\vartheta_{k'i}}(1-\theta_{k'i})
 \times\frac{\Pi_{j=m'+1}^{k'+1}(l_{k+1,j}-l_{ki})\Pi_{j=m'+1}^{k'-1}(l_{k-1,j}-l_{ki}+1)}
  {\Pi_{j=m'+1}^{k'} (l_{kj}-l_{ki})(l_{kj}-l_{ki}+1)}
\\
& \times \prod_{j\neq i,j=1}^{m'}\left(\frac{l_{kj}-l_{ki}+1}{l_{k-1,j}-l_{ki}+1}\right) \xi_{\La-\delta_{ki}} + \sum_{i=m'+1}^{k'}
\frac{\prod_{j=m'+1}^{k'-1}(l_{k-1,j}-l_{ki})}{\prod_{j\neq i,j =m'+1}^{k'}(l_{kj}-l_{ki})}
\xi_{\La-\delta_{ki}}.
\end{split}
\end{equation*}
Here $\vartheta_{k,i}=\theta_{k1}+\ldots+\theta_{k,i-1}+\theta_{k-1,i+1}+\ldots+\theta_{k-1,m'}$. The arrays $\La\pm \delta_{ki}$ are obtained from $\La$ by replacing $\la_{k'i}$ with $\la_{k'i}\pm1$.  We assume  that
$\xi_{\La}=0$ if the GT tableau $\La$ is not $\la/\mu$-admissible.\qed
\end{thm}

We use the shorthand notations $\mathscr E^\pm_{\La,ki}$ for the matrix elements involved in Theorem \ref{thm:gt-formula} as follows,
\beq\label{eq:matrix-elements}
e_k\xi_{\La}=\sum_{i=1}^{k'}\mathscr E^+_{\La,ki}\ \xi_{\La+\delta_{ki}},\quad 
f_k\xi_{\La}=\sum_{i=1}^{k'}\mathscr E^-_{\La,ki}\ \xi_{\La-\delta_{ki}}.
\eeq
Note that a different Gelfand-Tsetlin type basis for covariant representations of $\glMN$ is given in \cite[Theorem 4.18]{Mol10}.

\subsection{Super Yangian $\YglMN$}\label{sec rtt}
We recall the definition of super Yangian $\YglMN$ from \cite{Naz}, see also \cite{Gow07,Naz20} for some basic properties of $\YglMN$.

Let $\mc P\in \End(\mathcal V^{\otimes 2})$ be the $\Z_2$-graded flip operator,
\beq\label{eq:eijk}
\mathcal P=\sum_{i,j\in \bar I} s_jE_{ij}^{(1)} E_{ji}^{(2)},\quad \text{ where }\quad E_{ij}^{(1)}=E_{ij}\otimes 1,\quad E_{ij}^{(2)}=1\otimes E_{ij}.
\eeq
The super Yangian $\YglMN$ is the $\Z_2$-graded unital associative algebra with generators $\{t_{ij}^{(a)}\ |\ i,j\in \bar I, \, a\gge 1\}$ where the generators $t_{ij}^{(a)}$ have parities $|i|+|j|$. The defining relations of super Yangian $\YglMN$ are as follows. 

Define the Yang R-matrix $R(u)\in \End(\mathcal V^{\otimes 2})$ by $R(u)=1-\mathcal P/u$. Define the generating series $t_{ij}(u)\in \YglMN[[u^{-1}]]$ and $T_k(u)\in \YglMN[[u^{-1}]]\otimes \End(\mathcal V^{\otimes 2})$ by
\[
t_{ij}(u)=\delta_{ij}+\sum_{k=1}^\infty t_{ij}^{(k)}u^{-k},\quad T_k(u)=\sum_{i,j\in \bar I} (-1)^{|i||j|+|j|}E_{ij}^{(k)}\otimes t_{ij}(u),\quad k=1,2.
\]
Then defining relations of $\YglMN$ are written as
\beq\label{eq comm generators}
R(u_1-u_2)T_1(u_1)T_2(u_2)=T_2(u_2)T_1(u_1)R(u_1-u_2)\in \YglMN\otimes \End(\mathcal V^{\otimes 2}) [[u^{-1}]].
\eeq
In terms of generating series, defining relations \eqref{eq comm generators} are equivalent to
\beq\label{eq comm series}
(u_1-u_2)[t_{ij}(u_1),t_{kl}(u_2)]=(-1)^{|i||j|+|i||k|+|j||k|}(t_{kj}(u_1)t_{il}(u_2)-t_{kj}(u_2)t_{il}(u_1)).
\eeq
The super Yangian $\YglMN$ is a Hopf superalgebra with coproduct, antipode, counit given by
\beq\label{coproduct}
\Delta: t_{ij}(u)\mapsto \sum_{k\in \bar I} t_{ik}(u)\otimes t_{kj}(u),\qquad S: T(u)\mapsto T(u)^{-1},\qquad \varepsilon: T(u)\mapsto 1.
\eeq

For $z\in\C$, there exists an isomorphism of Hopf superalgebras,
\begin{align}
&\tau_z:\YglMN\to\YglMN, && t_{ij}(u)\mapsto t_{ij}(u-z).\label{eq tau z}
\end{align}
For any $\YglMN$-module $M$, denote by $M_z$ the $\YglMN$-module obtained by pulling back $M$ through the isomorphism $\tau_z$.

The universal enveloping superalgebra $\mathrm U(\glMN)$ is a subalgebra of $\YglMN$ via the embedding $e_{ij}\mapsto s_it_{ij}^{(1)}$. The left inverse of this embedding is the \emph{evaluation homomorphism} $\pi_{m|n}: \YglMN\to \UglMN$ given by
\beq\label{eq:evaluation-map}
\pi_{m|n}: t_{ij}(u)\mapsto \delta_{ij}+s_ie_{ij}u^{-1}.
\eeq
For any $\glMN$-module $M$, it is naturally a $\YglMN$-module obtained by pulling back $M$ through the evaluation homomorphism $\pi_{m|n}$. We denote the corresponding $\YglMN$-module by the same letter $M$ and call it an \emph{evaluation module}.

\subsection{Gauss decomposition and $\ell$-weights}\label{sec:Gauss}
The Gauss decomposition of $\YglMN$, see \cite{Gow07}, gives generating series
\[
e_{ij}(u)=\sum_{a\gge 1}e_{ij}^{(a)}u^{-a},\quad f_{ji}(u)=\sum_{a\gge 1}f_{ji}^{(a)}u^{-r},\quad d_k(u)=1+\sum_{a\gge 1}d_{k,a}u^{-a},
\]
where $1\lle i< j\lle m+n$ and $k\in \bar I$, such that
\begin{align*}
t_{ii}(u)&=d_i(u)+\sum_{k<i}f_{ik}(u)d_k(u)e_{ki}(u),\\
t_{ij}(u)&=d_i(u)e_{ij}(u)+\sum_{k<i}f_{ik}(u)d_k(u)e_{kj}(u),\\
t_{ji}(u)&=f_{ji}(u)d_i(u)+\sum_{k<i}f_{jk}(u)d_k(u)e_{ki}(u).
\end{align*}

For $i\in I$, let
\[
x_{i}^+(u)=\sum_{a\gge 1}x_{i,a}^{+}u^{-a}:=e_{i,i+1}(u),\quad x_{i}^-(u)=\sum_{a\gge 1}x_{i,a}^{-}u^{-a}:=f_{i+1,i}(u).
\]

The parities of $x_{i,r}^\pm$ are the same as that of $t_{i,i+1}^{(a)}$ while all $d_{k,a}$ are even. The super Yangian $\YglMN$ is generated by $x_{i,a}^{\pm}$ and $d_{k,a}$, where $i\in I$, $k\in \bar I$, and $a\gge 1$. The full defining relations are described in \cite[Lemma 4 or Theorem 3]{Gow07}. Here we only need the following relations in $\YglMN[[u^{-1},v^{-1}]]$,
\begin{subequations}\label{eq:reltationdx}
\begin{align}
&[d_i(u),d_k(v)]=0,\\
&(u-v)[d_i(u),x_j^+(v)]=(s_i\delta_{ij}-s_i\delta_{i,j+1})d_i(u)(x_j^+(v)-x_j^+(u)),\label{eq:dx+}\\	
&(u-v)[d_i(u),x_j^-(v)]=(s_i\delta_{ij}-s_i\delta_{i,j+1})(x_j^-(u)-x_j^-(v))d_i(u),\\
&[x_{j}^\pm(u),x_l^\pm(v)]=0 \text{ for }|j-l|>1,
\end{align}
\end{subequations}
for $i,k\in\bar I$ and $j,l\in I$. 

We call the commutative subalgebra generated by coefficients of $d_i(u)$, for all $i\in \bar I$, the {\it Gelfand-Tsetlin algebra} and denote it by $\mathrm{A}(\glMN)$. It is not hard to see that $\mathrm{A}(\glMN)$ is maximal commutative in $\YglMN$. We do not need this fact for the present paper.

Let $\gamma_1=0$ if $m>0$ and $\gamma_1=-1$ if $m=0$. Define $\gamma_k=\gamma_{k-1}+(s_{k-1}+s_k)/2$ recursively for $2\lle k\lle m+n$.

\begin{lem}[{\cite{Gow05}}]\label{lem:central}
	The coefficients of the series $\prod_{j\in\bar I} \big(d_{j}(u-\gamma_j)\big)^{s_j}$ are central in $\YglMN$.\qed
\end{lem}

Set $\mathcal B:=1+u^{-1}\C[[u^{-1}]]$ and $\mathfrak B:=\mathcal B^{\bar I}$. We call an element $\bm\zeta\in \mathfrak B$ an \emph{$\ell$-weight}. We write $\ell$-weights in the form $\bm\zeta=(\zeta_i(u))_{i\in \bar I}$, where $\zeta_i(u)\in \mathcal B$ for all $i\in \bar I$. 

Clearly $\mathfrak B$ is an abelian group with respect to the point-wise multiplication of the tuples. Let $\Z[\mathfrak B]$ be the group ring of $\mathfrak B$ whose elements are finite $\Z$-linear combinations of the form $\sum a_{\bm \zeta}[\bm\zeta]$, where $a_{\bm\zeta}\in \Z$.

Let $M$ be a $\YglMN$-module. We say that a nonzero vector $v\in M$ is \emph{of $\ell$-weight $\bm \zeta$} if $d_{i}(u)v=\zeta_i(u)v$ for $i\in \bar I$. We say that a vector $v\in M$ is a \emph{highest $\ell$-weight vector of $\ell$-weight $\bs\zeta$} if $v$ is of $\ell$-weight $\bs\zeta$ and $x_i^+(u)v=0$ for all $i\in I$. By the Gauss decomposition, one can deduce that $v$ is a highest $\ell$-weight vector of $\ell$-weight $\bs \zeta $ if and only if
\beq\label{eq:t-singular}
t_{ij}(u)v=0,\quad t_{kk}(u)v=\zeta_k(u)v,\quad 1\lle i<j\lle m+n,\ k\in \bar I.
\eeq

Let $M$ be a finite-dimensional $\YglMN$-module and $\bm\zeta\in \mathfrak B$ an $\ell$-weight. Let
\[
\zeta_i(u)=1+\sum_{j=1}^\infty \zeta_{i,j}u^{-j},\qquad \zeta_{i,j}\in \C.
\]
Denote by $M_{\bm \zeta}$ the \emph{generalized $\ell$-weight space} corresponding to the $\ell$-weight $\bm\zeta$,
\[
M_{\bm\zeta}:=\{v\in M~|~  (d_{i,j}-\zeta_{i,j})^{\dim M} v=0 \text{ for all }i\in \bar I, \ j\in \Z_{>0}\}.
\]
We call $M$ \emph{thin} if $\dim(M_{\bm\zeta})\lle 1$ for all $\bm\zeta\in\mathfrak B$. We call $M$ {\it tame} if the joint action of the Gelfand-Tsetlin algebra on $M$ is diagonalizable. In particular, if $M$ is thin, then $M$ is tame. 

For a finite-dimensional $\YglMN$-module $M$, define the $q$-{\it character} (or \emph{Gelfand-Tsetlin character}) of $M$ by the element 
\[
\chi(M):=\sum_{\bm\zeta\in \mathfrak B}\dim(M_{\bm \zeta})[\bm\zeta]\in \Z[\mathfrak B].
\]

Let $\mathcal C$ be the category of finite-dimensional $\YglMN$-modules. Let $\mathscr Rep(\mathcal C)$ be the Grothendieck ring of $\mathcal C$, then $\chi$ induces a $\Z$-linear map from $\mathscr Rep(\mathcal C)$ to $\Z[\mathfrak B]$.

\begin{lem}[{\cite[Lemma 2.8]{LM20}}]\label{lem chi morphism}
The map $\chi:\mathscr Rep(\mathcal C)\to \Z[\mathfrak B]$ is a ring homomorphism. \qed
\end{lem}

\subsection{Skew representations }
%

Let $\psi_{r}:\YglMN\to \mathrm Y(\gl_{m'|n})$ be the embedding given by
\[
\psi_{r}: d_i(u)\mapsto d_{i'}(u),\quad x_j^\pm(u)\mapsto x_{j'}^\pm(u),
\]
see \cite[Lemma 2]{Gow07}. 

Regard $\mathrm{Y}(\gl_{r})$ as a subalgebra of $\mathrm Y(\gl_{m'|n})$ via the natural embedding $t_{ij}(u)\mapsto t_{ij}(u)$, for $1\lle i,j\lle r$. Clearly, the subalgebra $\mathrm{Y}(\gl_{r})$ of $\mathrm Y(\gl_{m'|n})$ supercommutes with the image of $\YglMN$ under the map $\psi_{r}$, see \eqref{eq:reltationdx}. Therefore, the image of the homomorphism $$\pi_{m'|n}\circ \psi_{r}:\YglMN\to \mathrm{U}(\gl_{m'|n})$$ supercommutes with the subalgebra $\mathrm{U}(\gl_{r})$ in $\mathrm{U}(\gl_{m'|n})$. This implies that the subspace $L(\la/\mu)$ is invariant under the action of the image of $\pi_{m'|n}\circ \psi_{r}$. Therefore, $L(\la/\mu)$ is a $\YglMN$-module. We call $L(\la/\mu)$ a {\it skew representation}, see \cite[Section 3]{LM20} for more detail. The $q$-character of $L(\la/\mu)$ is computed in terms of semi-standard Young tableaux, see \cite[Theorem 3.4]{LM20}. In the rest of this section, we recompute the $q$-character of $L(\la/\mu)$ in terms of GT tablaux.

Define the series $\mathscr A_k(u)$, for $0\lle k\lle m+n$, in $\mathrm{Y}(\gl_{m'|n})[[u^{-1}]]$, by
\[
\mathscr A_k(u):=\prod_{i=1}^{r}d_{i}(u+r-i+1)\prod_{j=1}^{k}\big(d_{j'}(u-\gamma_j)\big)^{s_j}.
\]
For a $\la/\mu$-admissible Gelfand-Tsetlin tableau $\La$, define rational functions $\mathscr Y_{\La,k}(u)$, for $0\lle k\lle m+n$, by
\beq\label{eq:y-A-ction}
\begin{split}
\mathscr Y_{\La,k}(u)&=\prod_{i=1}^{r}\Big(\frac{u+\la_{k'i}+r-i+1}{u+r-i+1}\Big)\prod_{j=1}^{k}\Big(\frac{u+s_j\la_{k'j'}-\gamma_j}{u-\gamma_j}\Big)^{s_j}\\
&=\prod_{i=1}^{r}\Big(\frac{u+l_{ki}}{u+r-i+1}\Big)\prod_{j=1}^{k}\Big(\frac{u+l_{kj'}}{u-\gamma_j}\Big)^{s_j}.
\end{split}
\eeq
\begin{lem}\label{lem:A-action}
We have $\mathscr A_k(u)\xi_{\La}=\mathscr Y_{\La,k}(u)\xi_{\La}$ for $0\lle k\lle m+n$.	
\end{lem}
\begin{proof}
The proof is similar to \cite[Lemma 2.1]{NT98} and \cite[Lemma 4.7]{FM02} using Lemma \ref{lem:central}.
\end{proof}

Define rational functions $\zeta_{\La,k}(u)$, for $1\lle k \lle m+n$, by
\beq\label{eq:l-weight}
\zeta_{\La,k}(u)=\left(\frac{\mathscr Y_{\La,k}(u+\gamma_{k})}{\mathscr Y_{\La,k-1}(u+\gamma_{k})}\right)^{s_k}.
\eeq
and set $\bm{\zeta}_{\La}=(\zeta_{\La,1}(u),\dots,\zeta_{\La,m+n}(u))$.

\begin{lem}\label{lem:l-weight}
	The vector $\xi_{\La}\in L(\la/\mu)$ is of $\ell$-weight $\bm{\zeta}_{\La}$, namely $$d_k(u)\xi_{\La}=\zeta_{\La,k}(u)\xi_{\La},\qquad 1\lle k\lle m+n.$$
\end{lem}
\begin{proof}
Using the fact that 
\[
\varphi_{m'}(d_k(u))=(\mathscr A_k(u+\gamma_k)/\mathscr A_{k-1}(u+\gamma_k))^{s_k}, 
\]
the statement follows from Lemma \ref{lem:A-action}.
\end{proof}

Let $\beta=(\beta_1,\dots,\beta_{m+n})$ be a $\gl_{m|n}$-weight. Define a rational function $\mathscr Y_{\beta}(u)$ by
\[
\mathscr Y_{\beta}(u)=\prod_{i=1}^{m+n}\Big(\frac{u+s_i\beta_i-\gamma_i}{u-\gamma_i}\Big)^{s_i}.
\]
\begin{lem}\label{lem:for-thin}
Let $\beta$ and $\beta'$ be covariant $\gl_{m|n}$-weights. If $\mathscr Y_{\beta}(u)=\mathscr Y_{\beta'}(u)$, then $\beta=\beta'$.
\end{lem}
\begin{proof}
If $mn=0$, then the sequence $(s_i\beta_i-\gamma_i)_{i=1}^{m+n}$ is either strictly increasing or strictly decreasing. The statement is hence clear. We assume that $mn>0$.

We claim that if $\mathscr Y_{\beta}(u)=1$, then $\beta=(0,\dots,0)$; if $\mathscr Y_{\beta}(u)\ne 1$, then the smallest root of $\mathscr Y_{\beta}(u)$ is $-\beta_1$. Indeed, let $a$ be the smallest nonnegative integer such that $\beta_{a+1}=0$, then $\beta_i=0$ for all $i>a$ and $$\mathscr Y_{\beta}(u)=\prod_{i=1}^{a}\Big(\frac{u+s_i\beta_i-\gamma_i}{u-\gamma_i}\Big)^{s_i}.
$$

(1) If $a=0$, it is clear that $\beta=(0,\dots,0)$ and $\mathscr Y_{\beta}(u)=1$.

(2) If $1\lle a\lle m$, then $(s_i\beta_i-\gamma_i)_{i=1}^a$ is strictly decreasing. Moreover, $s_1\beta_1-\gamma_1>-\gamma_i$ and $s_i=1$ for $1\lle i\lle a$. Hence $-\beta_1$ ($=\gamma_1-\beta_1$ as $\gamma_1=0$) is the smallest zero of $\mathscr Y_{\beta}(u)$.

(3) If $a>m$, then to show that the smallest root of $\mathscr Y_{\beta}(u)$ is $-\beta_1$ it reduces to show that $s_1\beta_1-\gamma_1>s_i\beta_{i}-\gamma_{i}$ and $s_1\beta_1-\gamma_1>-\gamma_{i}$ for $i>1$. This is clear for $1\lle i\lle m$, see part (2). For $m+1<i\lle a$, note that $(s_i\beta_i-\gamma_i)_{i=m+1}^a$ is strictly increasing, it suffices to check that $\beta_1>-\beta_a-\gamma_a$ and $\beta_1>-\gamma_a$. Because $\beta$ is covariant, we have $$\beta_1\gge \beta_m\gge \#\{\beta_j>0|j>m+1\}=a-m>a-2m=-\gamma_a.$$

Now the claim follows. In particular, $\mathscr Y_\beta(u)$ uniquely determines $\beta_1$ and the rational function
\[
\prod_{i=2}^{m+n}\Big(\frac{u+s_i\beta_i-\gamma_i}{u-\gamma_i}\Big)^{s_i}.
\]
Since $(\beta_2,\dots,\beta_{m+n})$ is also a covariant $\gl_{m-1|n}$-weight, the above rational function uniquely determines $\beta_2$. Repeating this procedure, we conclude that $\mathscr Y_\beta(u)$ determines $\beta$ uniquely.
\end{proof}

\begin{rem}
	The strategy of proof is clear in terms of \cite[Theorem 3.4 and Lemma 3.6]{LM20} since the smallest zero of $\mathscr Y_{\beta}(u)$ corresponds to the largest content of the hook Young diagram corresponding to weight $\beta$ which is always given by the last box of the first row (assuming $m>0$).\qed
\end{rem}	

\begin{cor}\label{cor:thin}
	The skew representation $L(\la/\mu)$ is thin.
\end{cor}
\begin{proof}
Let $\La_1$ and $\La_2$ be $\la/\mu$-admissible. We show that if $\zeta_{\La_1,k}=\zeta_{\La_2,k}$ for all $1\lle k\lle m+n$, then $\La_1=\La_2$. Note that  $\mathscr Y_{\La,0}$ is independent of $\La$, therefore we have $\mathscr Y_{\La_1,k}=\mathscr Y_{\La_2,k}$ for all $0\lle k\lle m+n$. Note that each row of a $\la/\mu$-admissible GT tableau corresponds to a covariant weight of a certain general Lie superalgebra. The statement follows from Lemma \ref{lem:for-thin} with suitable choices of $m$ and $n$. Here we only remark that $\gamma_i$ should change correspondingly with respect to each choice.\end{proof}

\section{Main results for super Yangian}\label{sec main}
\subsection{Main results}
Our main result is the explicit matrix elements of Drinfeld generating series $x_{i}^{\pm}(u)$ and $d_i(u)$ with respect to the basis $\xi_{\La}$ for all $\la/\mu$-admissible Gelfand-Tsetlin tableaux $\La$. 

Define the integers $s_{ki}^\pm$ for $1\lle k\lle m+n$ and $1\lle i\lle k'$ by
\beq\label{eq:ski}
s_{ki}^+=\begin{cases}
\frac{s_k+1}{2}, & \text{ if } i\lle m';\\
-1, & \text{ if } i>m',	
\end{cases} \qquad s_{ki}^-=\begin{cases}
\frac{s_k-1}{2}, & \text{ if } i\lle m';\\
0, & \text{ if } i>m'.	
\end{cases}
\eeq
\begin{thm}\label{thm:main}
We have 
\[
d_k(u)\xi_{\La}=\zeta_{\La,k}(u)\xi_{\La},
\]
\[
x_k^{+}(u)\xi_{\La}=s_k\sum_{i=1}^{k'}\mathscr E_{\La,ki}^{+}\frac{\xi_{\La+\delta_{ki}}}{u+l_{ki}+s_{ki}^++\gamma_k},
\]
\[
x_k^{-}(u)\xi_{\La}=s_{k+1}\sum_{i=1}^{k'}\mathscr E_{\La,ki}^{-}\frac{\xi_{\La-\delta_{ki}}}{u+l_{ki}+s_{ki}^-+\gamma_k},
\]
where $l_{ki}$, $\mathscr E_{\La,ki}^{\pm}$, and $\zeta_{\La,k}(u)$ are defined in \eqref{eq:lki}, \eqref{eq:matrix-elements}, and \eqref{eq:l-weight}, respectively.
\end{thm}

\begin{thm}\label{thm:irr}
Every skew representation of $\YglMN$ is irreducible.	
\end{thm}

\begin{thm}\label{thm:thin-one}
Conjecture \ref{conj:main} is true for $m=n=1$.	
\end{thm}

We prove these theorems in the next three subsections. Note that Theorem 3.2 is obtained in \cite[Theorem 4.9]{LM20}. Here we give an independent proof using Theorem \ref{thm:main}.

\subsection{Proof of Theorem \ref{thm:main}}
We prepare the $\YglMN$ version of \cite[Proposition 3.1]{You15}. For each $i\in I$ and $a\in\C$, define the \emph{simple $\ell$-root} $A_{i,a}\in \mathfrak B$ by
\[
(A_{i,a})_j(u)=\frac{u-a}{u-a-(\alpha_i,\epsilon_j)},\qquad j\in \bar I.
\]

The following proposition established the property that all $x_i^\pm(u)$ acts on finite-dimensional representations of $\YglMN$ in a rather specific way.

\begin{prop}\label{prop:main}
Let $V$	be a finite-dimensional $\YglMN$-module. Pick and fix any $i\in I$. Let $(\bm \mu,\bm \nu)$ be a pair of $\ell$-weights of $V$ such that $x_{i,j}^{\pm}(V_{\bm\mu})\cap V_{\bm\nu}\ne \{0\}$ for some $j\gge 1$. Then:
\begin{enumerate}
\item $\bm \nu=\bm\mu A_{i,a}^{\pm 1}$ for some $a\in \C$,
\item there exist bases $(v_k)_{1\lle k\lle \dim(V_{\bm\mu})}$ of $V_{\bm\mu}$ and $(w_l)_{1\lle l\lle \dim(V_{\bm\nu})}$ of $V_{\bm\nu}$, and complex polynomials $P_{k,l}(z)$ of degree $\lle k+l -2$ such that
\[
(x_i^{\pm}(z)v_k)_{\bm \nu}=\sum_{l =1}^{\dim(V_{\bm \nu})} w_l P_{k,l}(\pa_z)\Big(\frac{1}{z-a}\Big),
\]
where $(x_i^\pm(z) v_k)_{\bm \nu}$ is the projection of $x_i^\pm(z) v_k$ onto the generalized $\ell$-weight subspace $V_{\bm \nu}$.
\end{enumerate}
\end{prop}
\begin{proof}
The proof is very similar to that of \cite{You15}. We give it here for completeness. 

Let $(v_k)_{1\lle k\lle \dim(V_{\bm\mu})}$  be a basis of $V_{\bm\mu}$ such that all $d_{j,r}$ act upper-triangularly. More precisely, for all $j\in \bar I$ and $1\lle k\lle \dim(V_{\bm\mu})$, we have 
\[
(d_j(u)-\mu_j(u)) v_k= \sum_{k'<k}v_{k'}\xi_j^{k,k'}(u),
\]
where $\xi_j^{k,k'}(u)$ are certain elements in $u^{-1}\C[[u^{-1}]]$. Similarly, let $(w_l)_{1\lle l\lle \dim(V_{\bm\nu})}$ be a basis of $V_{\bm\nu}$ such that all $d_{j,r}$ act lower-triangularly, namely for all $j\in \bar I$ and $1\lle l\lle \dim(V_{\bm \nu})$,
\[
(d_j(u)-\nu_j(u)) w_l= \sum_{l'<l}w_{l'}\zeta_j^{l,l'}(u),
\]
for certain $\zeta_j^{l,l'}(u)\in u^{-1}\C[[u^{-1}]]$.

We only show that statement for the case of $x_i^+(z)$. The case of $x_i^-(z)$ is similar.

For all $1\lle k\lle \dim(V_{\bm \mu})$, there exist formal series $\la_{k,l}(z)\in z^{-1}\C[[z^{-1}]]$ for every $l$, $1\lle l\lle \dim(V_{\bm \nu})$, such that
\[
(x_i^+(z) v_k)_{\bm \nu}=\sum_{l =1}^{\dim (V_{\bm \nu})}\la_{k,l}(z)w_{l}.
\]

It follows from \eqref{eq:dx+} that
\begin{align*}
(u-z)x_i^+(z)(d_j(u)-\mu_j(u)) v_k=(u-z-(\epsilon_j,\alpha_i))&\, d_j(u)x_i^+(z)v_k\\+&\ (\epsilon_j,\alpha_i)d_j(u)x_i^+(u)v_k-(u-z)x_i^+(z)\mu_j(u)v_k.
\end{align*}
Projecting the equation to $V_{\bm \nu}$ and taking the $w_l$ component, we obtain
\beq\label{eq:wl}
\begin{split}
(u-z)\sum_{k'=1}^{k-1}\xi_{j}^{k,k'}(u)\la_{k',l}(z)=&(u-z-(\epsilon_j,\alpha_i))\Big(\la_{k,l}(z)\nu_j(u)+\sum_{l'=1}^{l-1}\la_{k,l'}(z)\zeta_j^{l',l}(u)\Big)\\ 
+(\epsilon_j,\alpha_i)&\Big(\la_{k,l}(u)\nu_j(u)+\sum_{l'=1}^{l-1}\la_{k,l'}(u)\zeta_j^{l',l}(u)\Big)-(u-z)\mu_j(u)\la_{k,l}(z).
\end{split}
\eeq
Since $(x_i^+(z) V_{\bm\mu})_{\bm \nu}\ne 0$, there exists a smallest $k_0$ such that $(x_i^+(z)v_{k_0})_{\bm \nu}\ne 0$ and hence a smallest $l_0$ such that $\la_{k_0,l_0}(z)\ne 0$. Then \eqref{eq:wl} implies
\beq\label{eq:proofprop}
0=(u-z-(\epsilon_j,\alpha_i))\la_{k_0,l_0}(z)\nu_j(u)+(\epsilon_j,\alpha_i)\la_{k_0,l_0}(u)\nu_j(u)-(u-z)\mu_j(u)\la_{k_0,l_0}(z),
\eeq
for all $j\in\bar I$. Let 
\[
\la_{k_0,l_0}(z)=a_mz^{-m}+a_{m+1}z^{-m-1}+\cdots,
\]
where $m\gge 1$ and $a_m\ne 0$. Considering the coefficients of $z^{-m}$ in \eqref{eq:proofprop}, we have
\[
(u-(\epsilon_j,\alpha_i))a_m\nu_j(u)-a_{m+1}\nu_j(u)-a_mu\mu_j(u)+a_{m+1}\mu_j(u)=0.
\]
Thus 
$$
\nu_j(u)=\mu_j(u)\frac{u-a}{u-a-(\epsilon_j,\alpha_i)}=\mu_j(u)(A_{i,a})_j(u),
$$
where $a=a_{m+1}/a_m$ and part (1) follows.

Using part (1), equation \eqref{eq:wl} becomes
\begin{equation}\label{eq:induct}
\begin{split}
\frac{(\epsilon_j,\alpha_i)\mu_j(u)((u-a)\la_{k,l}(u)-(z-a)\la_{k,l}(z))}{u-a-(\epsilon_j,\alpha_i)}&\\	= (u-z)\Big(\sum_{k'=1}^{k-1}\xi_j^{k,k'}(u)\la_{k',l}(z)&-\sum_{l'=1}^{l-1}\la_{k,l'}(z)\zeta_j^{l',l}(u)\Big)\\+(\epsilon_j,\alpha_i)&\Big(\sum_{l'=1}^{l-1}\la_{k,l'}(z)\zeta_j^{l',l}(u)-\sum_{l'=1}^{l-1}\la_{k,l'}(u)\zeta_j^{l',l}(u)\Big).
\end{split}
\end{equation}
We show part (2) by induction on $k+l$. For the base case $k+l=2$, it follows from \eqref{eq:induct} that \[(u-a)\la_{1,1}(u)-(z-a)\la_{1,1}(z)=0.\]Note that $\la_{1,1}(z)\in z^{-1}\C[[z^{-1}]]$, it follows from direct computations that 
$$\la_{1,1}(z)=p_{1,1}\sum_{k\gge 1}a^kz^{-k-1}=\frac{p_{1,1}}{z-a}$$
for some $p_{1,1}\in \C$ which shows the base case. Suppose part (2) holds for all $(k',l')$ such that $k'+l'<k+l$. The case of $(k,l)$ follows directly by considering the coefficient of $u^{-1}$ in \eqref{eq:induct} and the inductive data.
\end{proof}

Let us finish the proof of Theorem \ref{thm:main}. We apply Proposition \ref{prop:main} for $V=L(\la/\mu)$, $\bm \mu=\bm\zeta_{\La}$, and $\bm\nu=\bm\zeta_{\La\pm \delta_{ki}}$. Hence it suffices to find all polynomials $P_{k,l}$ and the corresponding $a$.

The numbers $a$ are computed directly since the expressions for $\bm\zeta_{\La}$ and $\bm\zeta_{\La\pm \delta_{ki}}$ are known explicitly. More precisely, we have
\[
(A_{k,a})_k(u)=\frac{u-a}{u-a-s_k}=\frac{\zeta_{\La+\delta_{ki},k}(u)}{\zeta_{\La,k}(u)}=\begin{cases}
\left(\dfrac{u+l_{ki}+1+\gamma_k}{u+l_{ki}+\gamma_k}\right)^{s_k},&\text{ if }1\lle i\lle m';	\\
\left(\dfrac{u+l_{ki}-1+\gamma_k}{u+l_{ki}+\gamma_k}\right)^{-s_k},&\text{ if } i> m';
\end{cases}
\]
\[
(A_{k,a}^{-1})_k(u)=\frac{u-a-s_k}{u-a}=\frac{\zeta_{\La-\delta_{ki},k}(u)}{\zeta_{\La,k}(u)}=\begin{cases}
\left(\dfrac{u+l_{ki}-1+\gamma_k}{u+l_{ki}+\gamma_k}\right)^{s_k},&\text{ if }1\lle i\lle m';	\\
\left(\dfrac{u+l_{ki}+1+\gamma_k}{u+l_{ki}+\gamma_k}\right)^{-s_k},&\text{ if } i> m';
\end{cases}
\]
It is straightforward that $a=-l_{ki}-s_{ki}^\pm-\gamma_k$.

By Corollary \ref{cor:thin}, we know that $k=1$ and $l=1$. Hence $P_{1,1}$ in this case is a constant. Moreover, the constant $P_{1,1}$ is determined by the coefficient of $\xi_{\La,\pm\delta_{ki}}$ in $x_{k,1}^\pm\xi_{\La}$. It is easy to see from the Gauss decomposition that $x_{k,1}^+=t_{k,k+1}^{(1)}$ and $x_{k,1}^-=t_{k+1,k}^{(1)}$. Hence under our identification of $\glMN$ as a subalgebra of $\mathfrak{gl}_{m'|n}$ and using the evaluation map, we see that the action of $x_{k,1}^+$ and $x_{k,1}^-$ corresponds to that of $s_ke_k$ and $s_{k+1}f_k$, respectively, completing the proof of Theorem \ref{thm:main}.

\subsection{Proof of Theorem \ref{thm:irr}}Before starting the proof, we prepare several lemmas.

For a covariant $\gl_{m|n}$-weight $\beta$, let $l$ be the number of nonzero components among $\beta_{m+1},\dots,\beta_{m+n}$. We have $l\lle \beta_m$. There is an associated Young diagram $\Gamma_\beta$ whose first $m$ rows are $\beta_1,\dots,\beta_{m}$ while the first $l$ columns are $\beta_{m+1}+m,\dots,\beta_{m+l}+m$. Moreover, $\Gamma_\beta$ has no other boxes outside of its first $m$ rows and first $l$ columns. Note that the condition that $l\leqslant \beta_m$ ensures that $\Gamma_{\beta}$ corresponds to a partition. We say that $\Gamma_\beta$ is the \emph{$(m|n)$-hook Young diagram} associated with $\beta$.

Let $\Gamma_\la$ and $\Gamma_\mu$ be the corresponding $(m'|n)$-hook and $(r|0)$-hook Young diagrams, respectively. If $L(\la/\mu)$ is nontrivial, we must have $\Gamma_\mu\subset \Gamma_\la$. Hence we assume further that $\Gamma_\mu\subset \Gamma_\la$. Let $\Gamma_{\la/\mu}$ be the skew Young diagram $\Gamma_\la/ \Gamma_\mu$.

A \textit{semi-standard Young tableau} of shape $\Gamma_{\la/\mu}$ is the skew Young diagram $\Gamma_{\la/\mu}$ with an element from $\{1,2,\dots,m+n\}$ inserted in each box such that the following conditions are satisfied:
\begin{enumerate}
    \item the numbers in boxes are weakly increasing along rows and columns;
    \item the numbers from $\{1,2,\dots,m\}$ are strictly increasing along columns;
    \item the numbers from $\{m+1,m+2,\dots,m+n\}$ are strictly increasing along rows.
\end{enumerate}

 We also need the bijection between $\la/\mu$-admissible GT tableaux and semi-standard Young tableaux of shape $\Gamma_{\la/\mu}$. Let $\La$ be a $\la/\mu$-admissible GT tableau. Let $\la^{(k)}=(\la_{k'1},\dots,\la_{k'k'})$, then $\la^{(k)}$ is a covariant $\g_k$-weight, where $\g_k=\gl_{k'}$ if $k\lle m$ and $\g_k=\gl_{m'|k-m}$ if $k>m$. We have a chain of Young diagrams
 \[
 \Gamma_{\la}=\Gamma_{\la^{(m+n)}}\supset \Gamma_{\la^{(m+n-1)}}\supset \cdots\supset \Gamma_{\la^{(k)}}\supset \cdots\supset \Gamma_{\la^{(1)}}\supset \Gamma_{\la^{(0)}}=\Gamma_\mu.
 \]
 One obtains a semistandard Young tableau $\Omega_\La$ by inserting the number $k$ to the skew Young subdiagram $\Gamma_{\la^{(k)}}/\Gamma_{\la^{(k-1)}}$ for $1\lle k\lle m+n$. It is well known that the map $\La\to \Omega_\La$ defines a bijection between $\la/\mu$-admissible GT tableaux and semi-standard Young tableaux of shape $\Gamma_{\la/\mu}$, see \cite{BR87,SV10}.
 
 Given a $\la/\mu$-admissible GT tableau $\La$, if $\La\pm\delta_{ki}$ is also admissible, then we call the transformation from $\La$ to $\La\pm\delta_{ki}$ an \emph{admissible transformation}. In terms of semistandard Young tableaux,  it means the semistandard Young tableau corresponding to $\La+\delta_{ki}$ (resp. $\La-\delta_{ki}$) is obtained from the semistandard Young tableau corresponding to $\La$ by replacing one $k+1$ with $k$ (resp. one $k$ with $k+1$).
 
\begin{lem}\label{lem:adm-transform}
Let $\La$ and $\La'$ be $\la/\mu$-admissible, then one can obtain $\La'$ from $\La$ by several admissible transformations.
\end{lem}
\begin{proof}
It is easy to explain the proof using semistandard Young tableaux instead of GT tableaux. It suffices to show that one can obtain all semistandard Young tableaux of shape $\Gamma_{\la/\mu}$ from the semistandard Young tableau $\Omega^+$corresponding to the highest weight appearing in $L(\la/\mu)$. This could be easily done by transforming the columns of $\Omega^+$ to the desired ones from right to left and from bottom to top.
\end{proof}

\begin{lem}\label{lem:nonzero}
If $\La$ and $\La\pm\delta_{ki}$ are $\la/\mu$-admissible, then $\mathscr E^\pm _{\La,ki}$ is nonzero.
\end{lem}
\begin{proof}
The lemmas follows from a case-by-case computation. We give several examples for the case $\mathscr E^\pm _{\La,ki}$ when $m+1\lle k\lle m+n-1$.

Suppose that $\La$ and $\La+\delta_{ki}$ are $\la/\mu$-admissible. 

(1) If $i\lle m'$, then by A(2) we have $\theta_{k'i}=1$ and $\theta_{k'-1,i}=0$. Now it reduces to show that $l_{kj}-l_{ki}-1\ne 0$ for $j\lle m'$. It follows from A(5) that the sequence $(l_{kj})_{j=1}^{m'}$ is strictly decreasing. Hence we only need to check that $l_{k,i-1}-l_{ki}-1\gge 1$, namely $\la_{k',i-1}\gge \la_{k'i}+1$. This follows from A(5) as $\La+\delta_{ki}$ is admissible. 

(2) If $i>m'$, we show that $l_{kj}-l_{ki}\ne 0$ for $j\lle m'$. Since $\La+\delta_{ki}$ is admissible, it follows from A(3) that $\la_{k'm'}\gge i-m'$. Hence
\[
l_{kj}-l_{ki}=\la_{k'j}-j+1-(-\la_{k'i}+i-2m')\gge i-m'+j+1-i+2m'=m'+1-j>0.
\]

Suppose that $\La$ and $\La-\delta_{ki}$ are $\la/\mu$-admissible. We elaborate more for the case $\mathscr E_{ki}^-$ for the factor
\beq\label{eq:factor}
 \frac{\prod_{j=m'+1}^{k'+1}(l_{k+1,j}-l_{ki})\prod_{j=m'+1}^{k'-1}(l_{k-1,j}-l_{ki}+1) }{\prod_{j=m'+1}^{k'}(l_{kj}-l_{ki})(l_{kj}-l_{ki}+1)},
\eeq
where $i\lle m'$.

Let $\la_{k'+1,m'}=a$, $\la_{k'm'}=b$, $\la_{k'-1,m'}=c$, then $b+1\gge a\gge b$ and $c+1\gge b\gge c$. 

(1) If $i<m'$, then by A(3), the factor \eqref{eq:factor} is equal to
\begin{align*}
	&\frac{\prod_{j=m'+1}^{m'+a}(l_{k+1,j}-l_{ki})\prod_{j=m'+a+1}^{k'+1}(j'-2m'-l_{ki})}{\prod_{j=m'+1}^{m'+b}(l_{kj}-l_{ki})\prod_{j=m'+b+1}^{k'}(j'-2m'-l_{ki})}\\
	&\qquad\qquad\qquad\qquad\times \frac{\prod_{j=m'+1}^{m'+c}(l_{k-1,j}-l_{ki}+1)\prod_{j=m'+c+2}^{k'}(j'-2m'-l_{ki})}{\prod_{j=m'+1}^{m'+b}(l_{kj}-l_{ki}+1)\prod_{j=m'+b+2}^{k'+1}(j'-2m'-l_{ki})}. 
\end{align*}
We only care about the part
\beq\label{eq:cancel}
\frac{\prod_{j=m'+a+1}^{k'+1}(j'-2m-l_{ki})\prod_{j=m'+c+2}^{k'}(j'-2m'-l_{ki})}{\prod_{j=m'+b+1}^{k'+1}(j'-2m'-l_{ki})\prod_{j=m'+b+2}^{k'}(j'-2m'-l_{ki})}=\frac{\prod_{j=m'+c+2}^{m'+b+1}(j'-2m'-l_{ki})}{\prod_{j=m'+b+1}^{m'+a}(j'-2m'-l_{ki})},
\eeq
since one easily checks similarly to the previous cases that the rests are nonzero. If $a+c=2b$, then everything in \eqref{eq:cancel} cancels out. If $b=c+1$ and $a=b$, then \eqref{eq:cancel} becomes
\[
m'+b'+1-2m'-l_{ki}<  m'+b'+1-2m'-l_{km'}=0.
\]
The case of $a=b+c$ and $b=c$ is similar.

(2) If $i=m'$, then $a=b=c+1$. By A(2) and A(3),  the factor \eqref{eq:factor} is equal to
\begin{align*}
&\frac{\prod_{j=m'+1}^{m+a'}(l_{k+1,j}-l_{ki})\prod_{j=m'+a+1}^{k'+1}(j'-2m'-l_{ki})}{\prod_{j=m'+1}^{m'+b-1}(l_{kj}-l_{ki})\prod_{j=m'+b}^{k'}(j'-2m'-l_{ki})}	\\
&\qquad\qquad \qquad\qquad \times\frac{\prod_{j=m'+1}^{m'+c}(l_{k-1,j}-l_{ki}+1)\prod_{j=m'+c+2}^{k'}(j'-2m'-l_{ki})}{\prod_{j=m'+1}^{m'+b-1}(l_{kj}-l_{ki}+1)\prod_{j=m'+b+1}^{k'+1}(j'-2m'-l_{ki})}.
\end{align*}
Again, we only care about the part
$$
\frac{\prod_{j=m'+a+1}^{k'+1}(j'-2m'-l_{ki})\prod_{j=m'+c+2}^{k'}(j'-2m'-l_{ki})}{\prod_{j=m'+b}^{k'+1}(j'-2m'-l_{ki})\prod_{j=m'+b+1}^{k'}(j'-2m'-l_{ki})}=\frac{1}{m'+b'-2m'-l_{ki}}=-1.
$$

The rest cases follow from similar straightforward computations.
\end{proof}

Now we will prove Theorem \ref{thm:irr}. Since the actions of all $d_i(u)$ on $L(\la/\mu)$ are simultaneously diagonalizable and have simple spectrum, see Corollary \ref{cor:thin}, it suffices to show that we can get all possible $\xi_{\La'}$ for $\la/\mu$-admissible $\La'$ from $\xi_{\La}$ for any given $\la/\mu$-admissible $\La$. 

By Theorem \ref{thm:main} and Lemma \ref{lem:adm-transform}, it suffices to show that all $\xi_{\La\pm\delta_{ki}}$ with $\la/\mu$-admissible $\La\pm\delta_{ki}$ are elements in $\YglMN \xi_{\La}$ for $j\in \Z_{> 0}$. 

By Theorem \ref{thm:main}, we have
\[
\sum\mathscr E_{\La,ki}^{\pm}\frac{\xi_{\La\pm\delta_{ki}}}{u+l_{ki}+s_{ki}^\pm+\gamma_k}\in \C x_k^{\pm}(u)\xi_{\La},
\]
where the summation is over all $i$ with $\la/\mu$-admissible $\La\pm\delta_{ki}$. Since $L(\la/\mu)$ is thin by Corollary \ref{cor:thin}, namely all $\xi_{\La\pm\delta_{ki}}$ correspond to different $\ell$-weights, it follows from Lemma \ref{lem:nonzero} that $\xi_{\La\pm\delta_{ki}}\in \YglMN\xi_{\La}$, completing the proof of Theorem \ref{thm:irr}. 

\subsection{Tame modules of $\Yone$}\label{sec:Yone}
In this section, we study tame modules of $\Yone$ and prove Theorem \ref{thm:thin-one}. We start by collecting some equalities in $\Yone$ that will be used, see e.g. \cite[Appendix]{HLM}\footnote{Note that $t_{ij}(u)$ corresponds to $\mathcal L_{ji}(u)$ therein.}.

By \eqref{eq comm series}, we have
\beq\label{eq-last-1}
t_{11}(u)t_{21}(x)=\frac{u-x-1}{u-x}t_{21}(x)t_{11}(u)+\frac{1}{u-x}t_{21}(u)t_{11}(x).
\eeq
Differentiating both sides with respect to $x$, we obtain
\beq\label{eq-last-2}
\begin{split}
t_{11}(u)t_{21}'(x)=\frac{u-x-1}{u-x}t_{21}'(x)t_{11}(u)\ +&\ \frac{1}{u-x}t_{21}(u)t_{11}'(x)\\-&\ \frac{1}{(u-x)^2}(t_{21}(x)t_{11}(u)-t_{21}(u)t_{11}(x)).
\end{split}
\eeq

By Gauss decomposition, we have $d_1(u)=t_{11}(u)$ and 
\[
d_2(u)=t_{22}(u)-t_{21}(u)(t_{11}(u))^{-1}t_{12}(u).
\]
The coefficients of $d_1(u)(d_2(u))^{-1}$ are central in $\mathrm{Y}(\gl_{1|1})$, see Lemma \ref{lem:central}.

\begin{lem}\label{lem last-3}
We have $t_{12}(u+1)t_{21}(u)=-t_{22}(u+1)t_{11}(u)+d_2(u)(d_1(u))^{-1}t_{11}(u)t_{11}(u+1)$.	
\end{lem}
\begin{proof}
By \eqref{eq comm generators}, we have $t_{12}(u)t_{11}(u+1)=t_{11}(u)t_{12}(u+1)$ and 
\[
[t_{12}(u+1),t_{21}(u)]=-t_{22}(u+1)t_{11}(u)+t_{22}(u)t_{11}(u+1).
\]
Therefore, 
\begin{align*}
t_{12}(u+1)t_{21}(u)=&-t_{22}(u+1)t_{11}(u)+t_{22}(u)t_{11}(u+1)-t_{21}(u)t_{12}(u+1)\\
=& 	-t_{22}(u+1)t_{11}(u)+\big(t_{22}(u)-t_{21}(u)(t_{11}(u))^{-1}t_{12}(u)\big)t_{11}(u+1)\\
=& -t_{22}(u+1)t_{11}(u)+d_2(u)(d_1(u))^{-1}t_{11}(u)t_{11}(u+1).\qedhere
\end{align*}
\end{proof}

\medskip

We recall some basic facts about representations of $\Yone$.

Let $\bm \zeta=(\zeta_1(u),\zeta_2(u))$, where $\zeta_1(u),\zeta_2(u)\in\mathcal B=1+u^{-1}\C[[u^{-1}]]$. Denote $L(\bm\zeta)$ the irreducible $\Yone$-module generated by a highest $\ell$-weight vector $v^+$ of $\ell$-weight $\bm \zeta$. It is known from \cite[Theorem 4]{Z95} that $L(\bm\zeta)$ is finite-dimensional if and only if
\beq\label{eq last 9}
\frac{\zeta_1(u)}{\zeta_2(u)}=\frac{\varphi(u)}{\psi(u)},
\eeq
where $\varphi$ and $\psi$ are relatively prime polynomials in $u$ of the same degree. Set $\deg\varphi=k$, then it also known that $\dim L(\bm\zeta)=2^k$, see \cite[Theorem 4]{Z95}.

Let
\beq\label{eq last 8}
\varphi(u)=\prod_{i=1}^k(u+a_i),\qquad \psi(u)=\prod_{j=1}^k(u-b_j),
\eeq
where $a_i,b_j\in\C$. Then $a_i+b_j\ne 0$ for all $1\lle i,j\lle k$.

For a subset $J$ of $\{1,\dots,k\}$, set $\varphi_{J}=\prod_{i\in J } (u+a_i)$. By convention, $\varphi_{\emptyset}=1$.

\begin{lem}\label{lem:char-1-1}
We have 
$$
\chi(L(\bm\zeta))=\sum_{J\subset \{1,\dots,k\}}[\bm\zeta]\cdot\left[\Big(\frac{\varphi_{J}(u-1)}{\varphi_{J}(u)},\frac{\varphi_{J}(u-1)}{\varphi_{J}(u)}\Big)\right],
$$	
where the summation is over all subsets of $\{1,\dots,k\}$.
\end{lem}
\begin{proof}
Recall that $L(a_i,b_i)$ is the two-dimensional irreducible $\gl_{1|1}$-module with the highest weight $(a_i,b_i)$. We also have the evaluation $\Yone$-module $L(a_i,b_i)$. Clearly, up to a one-dimensional module, we have 
$$
L(\bm\zeta)\cong \bigotimes_{i=1}^k L(a_i,b_i),
$$
where the order for the tensor product is not important as the tensor product is irreducible. Note that under the condition $a_i+b_j\ne 0$ for all $1\lle i,j\lle k$, the irreducibility of the tensor product $\bigotimes_{i=1}^k L(a_i,b_i)$ follows from either \cite[Theorem 5]{Z95} or \cite[Theorem 4.2]{Zha16}. Hence it suffices to consider irreducible tensor products of the form above.

Using the Gauss decomposition, one has
\[
\chi(L(a_i,b_i))=\left[\Big(1+\frac{a_i}{u},1-\frac{b_i}{u}\Big)\right]\left(1+\left[\Big(\frac{u+a_i-1}{u+a_i},\frac{u+a_i-1}{u+a_i}\Big)\right]\right).
\]
Now the statement follows from the fact that the $q$-character map is a homomorphism from $\mathcal C$ to $\Z[\mathfrak B]$, see Lemma \ref{lem chi morphism}.
\end{proof}

\begin{cor}\label{cor last}
The $\Yone$-module $L(\bm\zeta)$ is thin if and only if $\varphi$ has no multiple roots.\qed	
\end{cor}

Indeed, the condition that $\varphi$ has no multiple roots is also a necessary condition for $L(\bm\zeta)$ being tame.

\begin{prop}\label{prop last}
The $\Yone$-module $L(\bm\zeta)$ is tame if and only if $\varphi$ has no multiple roots.	
\end{prop}
\begin{proof}
	Again it suffices to show ``only if" part for the case of $\bigotimes_{i=1}^k L(a_i,b_i)$. We adopt the method from \cite{NT98}. Note that in this case, a  finite-dimensional irreducible module is tame if and only if the action of $t_{11}(u)$ on it is semisimple.
	
	Suppose $a_1$ is a multiple root of $\varphi$, we show that $\bigotimes_{i=1}^k L(a_i,b_i)$ is not tame. Without loss of generality, we assume there exists $\eta$ such that $a_i=a_1-1$ for $\eta< i\lle k$ and $a_j\ne a_1-1$ for $1\lle  j\lle \eta$.
	
	By Lemma \ref{lem:char-1-1}, there exists an $\ell$-weight vector $v$ of $\ell$-weight 
	$$
	\Big(u^{-k}\prod_{i=1}^\eta(u+a_i)\prod_{i=\eta+1}^k(u+a_i-1),u^{-k}\prod_{i=1}^k(u-b_i)\prod_{i=\eta+1}^k\frac{u+a_i-1}{u+a_i}\Big).
	$$
	Denote the polynomial $\prod_{i=1}^\eta(u+a_i)\prod_{i=\eta+1}^k(u+a_i-1)$ by $\wp(u)$. We have $\wp(-a_i)=\wp'(-a_i)=0$ by assumption.
	
	Let $\tl t_{ij}(u)$ be the linear operators corresponding to $u^k t_{ij}(u)$ in $\End(\bigotimes_{i=1}^k L(a_i,b_i))((u^{-1}))$, respectively. Then all $\tl t_{ij}(u)$ are polynomials in $u$, see \eqref{coproduct}. In particular, we have $\tl t_{11}(u)v=\wp(u)v$.
	
	Set $w=\tl t_{21}(-a_1)v$ and $w'=\tl t_{21}'(-a_1)v$, where $\tl t_{21}'(u)$ is the derivative of $\tl t_{21}(u)$ with respect to $u$. It follows from \eqref{eq-last-1} and $\tl t_{11}(-a_1)v=\wp(-a_1)v=0$ that
	\[
	\tl t_{11}(u)w=\tl t_{11}(u)\tl t_{21}(-a_1)v=\frac{u+a_1-1}{u+a_1}\wp(u)w.
	\]
	Similarly, by \eqref{eq-last-2} we have
	\[
	\tl t_{11}(u)w'=\tl t_{11}(u)\tl t_{21}'(-a_1)v=\wp(u)\left(\frac{u+a_1-1}{u+a_1}w'-\frac{1}{(u+a_1)^2}w\right).
	\]
	Hence it suffices to show that $w\ne 0$.
	
	By Lemma \ref{lem:central} and Lemma \ref{lem last-3}, we have
	\begin{align*}
	\tl t_{12}(u+1)\tl t_{21}(u)v
	=& -\tl t_{22}(u+1)\wp(u)v+\Big(\prod_{i=1}^k\frac{u-b_i}{u+a_i}\Big)\wp(u)\wp(u+1)v\\
	=& -\tl t_{22}(u+1)\wp(u)v+\prod_{i=1}^\eta(u+a_i+1)\prod_{i=\eta+1}^k(u+a_i-1)\prod_{i=1}^k(u-b_i)v.
	\end{align*}
    Setting $u=-a_1$, one obtains
    \[
    \tl t_{12}(-a_1+1)w=\prod_{i=1}^\eta(-a_1+a_i+1)\prod_{i=\eta+1}^k(-a_1+a_i-1)\prod_{i=1}^k(-a_1-b_i)w\ne 0.
    \]
    Here we used the definition of $\eta$.	
\end{proof}

\begin{proof}[Proof of Theorem \ref{thm:thin-one}]
	It follows from Corollary \ref{cor last} and Proposition \ref{prop last}.
\end{proof}	

\begin{rem}
We remark that to determine if a $\Yone$-module is tame, it depends on the choice of the parity sequence in the RTT presentation (or the choice of Borel subalgebra of $\Yone$). This is a new feature for super case, cf. \cite[Proposition 1.14]{NT98}. If the parity sequence is $|0|=\bar 1$ and $|1|=\bar 0$ instead, we denote the super Yangian corresponding to this new choice by $\wt{\mathrm Y}(\gl_{1|1})$. One can define highest $\ell$-weight theory and $q$-characters for $\wt{\mathrm Y}(\gl_{1|1})$. We use tilde (not to be confused with the tilde used in the proof above) to distinguish notations from $\Yone$ and $\wt{\mathrm Y}(\gl_{1|1})$. Clearly, $\Yone$ and $\wt{\mathrm Y}(\gl_{1|1})$ are isomorphic and one can identify $\tl t_{ij}(u)$ of $\wt{\mathrm Y}(\gl_{1|1})$ with $t_{3-i,3-j}(u)$ of $\Yone$. Note that we have
$$
d_1(u)\big(d_2(u)\big)^{-1}=\big(\tl d_{1}(u+1)\big)^{-1}\tl d_2(u+1),
$$
see e.g. \cite[Section 4.3]{LM19}. Consider the $\Yone$-module $L(\bm\zeta)$. As a $\wt{\mathrm Y}(\gl_{1|1})$-module, $L(\bm\zeta)$ has the highest $\ell$-weight $\bm{\tl\zeta}=(\tl\zeta_1(u),\tl \zeta_2(u))$ satisfying 
$$
\frac{\tl\zeta_1(u)}{\tl \zeta_2(u)}=\frac{\zeta_2(u-1)}{\zeta_1(u-1)}=\frac{\psi(u-1)}{\varphi(u-1)}.
$$
Similarly, one shows that $L(\bm\zeta)$ is tame with respect to $\wt{\mathrm Y}(\gl_{1|1})$ if and only if $\psi$ has no multiple roots. \qed
\end{rem}

We conclude this section by an example which we mentioned in the introduction.

\begin{eg}
Consider the tensor product of evaluation modules $M:=L(3,0)\otimes L(-1,0)$. Its highest $\ell$-weight is $(\frac{(u+3)(u-1)}{u^2},1)$. The corresponding polynomial $\varphi(u)$ is $(u+3)(u-1)$ which has no multiple roots. Therefore, 
the module $M$ is thin. Moreover, $M\cong L(2,0)\oplus L(1,1)$ as $\gl_{1|1}$-modules. However, $M$ is not isomorphic to any tensor products of skew representations as the highest $\ell$-weights $\bm \zeta=(\zeta_1,\zeta_2)$ of skew representations satisfy $a_i+b_i\gge 0$ for all $i$ after rearranging $b_i$, where $a_i$ and $b_i$ are defined in \eqref{eq last 9} and \eqref{eq last 8}.\qed
\end{eg}

\section{Quantum affine superalgebra}\label{sec q}
Throughout this section, we shall assume $q\in\C^\times$ is generic. Recall $s_i$ from Section \ref{sec glmn} and set $q_i=q^{s_i}$. For $k\in\Z$, we also set $$[k]:=\frac{q^k-q^{-k}}{q^{}-q^{-1}}.$$ We call $[k]$ the \emph{associated $q$-number} of $k$. Sometimes, we shall use the same notations from previous sections for the counter part of quantum case. 

\subsection{Quantum superalgebra $\UqglMN$} The quantum superalgebra $\UqglMN$ is a superalgebra with generators $e_{i}^\pm$ and $t_j^\pm$ for $i\in I$ and $j\in \bar I$ where $|e_m^\pm|=\bar 1$ and $|t_i^{\pm 1}|=|e_j^\pm|=\bar 0$ for $i\in \bar I$ and $j\in I\setminus\{m\}$. The relations of $\UqglMN$ \cite[Proposition 10.4.1]{Yam94} are given by
\[
t_it_i^{-1}=t_i^{-1}t_i=1,\qquad t_ie_j^\pm t_i^{-1}=q_i^{(\epsilon_i,\alpha_j)}e_j^\pm,
\]
\[
[e_j^+,e_k^-]=\delta_{jk}\frac{t_j^{s_j}t_{j+1}^{-s_{j+1}}-t_j^{-s_j}t_{j+1}^{s_{j+1}}}{q_j-q_j^{-1}},
\]
\[
[e_i^\pm,e_j^\pm]=(e_m^\pm)^2=0,\quad \text{ if }|i-j|\ne 1,
\]
\[
[e_{j}^\pm,[e_j^\pm,e_{j+1}^\pm]_{q^{-1}}]_q=[e_{j}^\pm,[e_j^\pm,e_{j-1}^\pm]_{q^{-1}}]_q=0,\quad \text{ if }j\in I\setminus \{m\},
\]
\[
[[[e_{m-1}^\pm,e_m^\pm]_q,e_{m+1}^\pm]_{q^{-1}},e_m^\pm]=0,\quad \text{ when }m,n>1,
\]
where $[a,b]_q=ab-(-1)^{|a||b|}qba$ for homogeneous $a,b$ and $q\in\C$, and $[\cdot,\cdot]=[\cdot,\cdot]_{1}$.

Let $\la=(\la_1,\la_2,\dots,\la_{m+n})$ be a tuple of complex numbers. Denote $\mathscr L(\la)$ the irreducible module generated by a nonzero vector $v$ satisfying the conditions 
\[
t_{i}v=q^{\la_i} v,\quad e_{j}v=0,
\]
for $i\in \bar I$ and $j\in I$. We call $\mathscr L(\la)$ the {\it irreducible $\UqglMN$-module of highest weight $\la$}.

Let $\mathscr V:=\C^{m|n}$ be the vector superspace with a basis $v_i$, $i\in \bar I$, such that $|v_i|=|i|$. Let $E_{ij}\in\End(\mathscr V)$ be the linear operators such that $E_{ij}v_k=\delta_{jk}v_i$. The map $\rho_{\mathscr V}:\UqglMN\to \End(\mathscr V)$,
$$
\rho_{\mathscr V}(t_i)=\sum_{k\in \bar I}q^{\delta_{ik}}E_{kk},\quad \rho_{\mathscr V}(e_j^+)=E_{j,j+1},\quad \rho_{\mathscr V}(e_j^-)=E_{j+1,j},\quad i\in \bar I,~ j\in I,
$$ 
defines a $\UqglMN$-module structure on $\mathscr V$. We call it the \emph{vector representation} of $\UqglMN$. The highest weight of $\mathscr V$ is the tuple $(1,0,\dots,0)$. 

The following theorem was shown for essentially typical $\glMN$-weight $\la$ in \cite[Theorem 4]{PSV94}. However, the same proof also works for covariant $\la$, cf. \cite[Theorem 7]{SV10} and \cite{FSZ20}.
\begin{thm}[{\cite[Theorem 4]{PSV94}}]\label{thm:q-GT-formulas}
Let $\la$ be an $(m|n)$-covariant weight, then $\mathscr L(\la)$
admits a basis $\xi_{\La}$ parameterized by all $\la$-admissible GT tableaux $\La$. The actions of the generators of $\UqglMN$ are given by the formulas
\begin{equation*}
t_{k}\xi_{\La}=q^{\left(\sum_{i=1}^{k}\la_{kj}-\sum_{j=1}^{k-1}\la_{k-1,j}\right)}\xi_{\La},\quad 1\lle k\lle m+n;
\end{equation*}
\begin{equation*}
e_{k}^+\xi_{\La}=-\sum_{i=1}^{k}\frac{\Pi_{j=1}^{k+1}[l_{k+1,j}-l_{ki}] }
  {\Pi_{j\neq i,j=1}^{k} [l_{kj}-l_{ki}] }\xi_{\La+\delta_{ki}},
\quad 1\lle k\lle m-1;
\end{equation*}

\begin{equation*}
e_{k}^-\xi_{\La}=\sum_{i=1}^{k}\frac{\Pi_{j=1}^{k-1}[l_{k-1,j}-l_{ki}]}
{\Pi_{j\neq i,j=1}^{k}[l_{kj}-l_{ki}]}\xi_{\La-\delta_{ki}},
\quad  1\lle k\lle m-1;
\end{equation*}
\begin{equation*}
\begin{split}
e_{m}^+\xi_{\La}=&\sum_{i=1}^{m}\theta_{mi}(-1)^{i-1}(-1)^{\theta_{m1}+\cdots+\theta_{m,i-1}}
\\
&\times  \frac{\Pi_{1\lle j< i\lle m} [l_{mj}-l_{mi}-1]}
  {\Pi_{1\lle i<j\lle m} [l_{mj}-l_{mi}]
    \Pi_{j\neq i,j=1}^{m}[l_{m+1,j}-l_{mi}-1]}
    \xi_{\La+\delta_{mi}},
\end{split}
\end{equation*}

\begin{equation*}
\begin{split}
e_{m}^- \xi_{\La}=&\sum_{i=1}^{m}(1-\theta_{mi})(-1)^{i-1}(-1)^{\theta_{m1}+\cdots+\theta_{m,i-1}}
\\
&\times  \frac{[l_{mi}-l_{m+1,m+1}]\Pi_{1\lle  i<j\lle m} [l_{mj}-l_{mi}+1]\Pi_{j=1}^{m-1}[l_{m-1,j}-l_{mi}]}
  {\Pi_{1\lle j< i\lle m} [l_{mj}-l_{mi}]} \xi_{\La-\delta_{mi}},
\end{split}
\end{equation*}
and for $m+1\lle k\lle m+n-1$
\begin{equation*}
\begin{split}
e_k^+ \xi_{\La}=&\sum_{i=1}^{m}\theta_{ki}(-1)^{\vartheta_{ki}}(1-\theta_{k-1,i})\times
\prod_{j\neq i,j =1}^{m}\left(\frac{[l_{kj}-l_{ki}-1]}{[l_{k+1,j}-l_{ki}-1]}\right)
\xi_{\La+\delta_{ki}}
\\
&-\sum_{i=m+1}^{k} 
\prod_{j=1}^{m}\left(\frac{[l_{kj}-l_{ki}][l_{kj}-l_{ki}+1]}{[l_{k+1,j}-l_{ki}][l_{k-1,j}-l_{ki}+1]} \right)\times
  \frac{\Pi_{j=m+1}^{k+1}[l_{k+1,j}-l_{ki}]}
  {\Pi_{j\neq i,j=m+1}^{k} [l_{kj}-l_{ki}]}\xi_{\La+\delta_{ki}} ,
\end{split}
\end{equation*}

\begin{equation*} 
\begin{split}
e_{k}^-\xi_{\La}=&
\sum_{i=1}^{m}\theta_{k-1,i}(-1)^{\vartheta_{ki}}(1-\theta_{ki})\times\frac{\Pi_{j=m+1}^{k+1}[l_{k+1,j}-l_{ki}]\Pi_{j=m+1}^{k-1}[l_{k-1,j}-l_{ki}+1]}
  {\Pi_{j=m+1}^{k} [l_{kj}-l_{ki}][l_{kj}-l_{ki}+1]}
\\
&\times
\prod_{j\neq i,j=1}^{m}\left(\frac{[l_{kj}-l_{ki}+1]}{[l_{k-1,j}-l_{ki}+1]}\right)
\xi_{\La-\delta_{ki}}
  + \sum_{i=m+1}^{k}
\frac{\prod_{j=m+1}^{k-1}[l_{k-1,j}-l_{ki}]}{\prod_{j\neq i,j =m+1}^{k}[l_{kj}-l_{ki}]}
\xi_{\La-\delta_{ki}}.
\end{split}
\end{equation*}
Here we use the same notations and conventions as in Theorem \ref{thm:gt-formula} with $r=0$. Note that $k'=k$ for all $k\in\Z$.\qed
\end{thm}

\subsection{Quantum affine superalgebra $\UqglMNh$}
Let $\mathscr R(u)\in\End(\mathscr V^{\otimes2})[u]$ be the trigonometric R-matrix, see e.g. \cite{PS81},
\beq\label{eq:q-R-matrix}
\begin{split}
	\mathscr R(u)=&\sum_{i\in \bar I}(uq_i-q_i^{-1})E_{ii}\otimes E_{ii}+(u-1)\sum_{i\ne j}E_{ii}\otimes E_{jj}\\
	&\qquad \qquad+u\sum_{i<j}(q_i-q_i^{-1})E_{ji}\otimes E_{ij}+\sum_{i<j}(q_j-q_j^{-1})E_{ij}\otimes E_{ji}.
\end{split}
\eeq

The quantum affine superalgebra $\UqglMNh$ is the unital associative superalgebra generated by $t_{ij}^{(a,\pm)}$ for $i,j\in\bar I$ and $a\in\Z_{\gge 0}$ with parity $|i|+|j|$. For $k=1,2$, set 
\[
T_k^\pm(u)=\sum_{i,j\in\bar I} t_{ij}^\pm(u)\otimes E_{ij}^{(k)}\in \UqglMNh\otimes \End(\mathscr V^{\otimes 2})[[u^{\pm 1}]],
\]
where $E_{ij}^{(k)}$ are defined similarly as in \eqref{eq:eijk} and
\[
t_{ij}^\pm(u)=\sum_{a\gge 0}t_{ij}^{(a,\pm)}u^{\pm a}\in \UqglMNh[[u^{\pm 1}]].
\]
The relations of $\UqglMNh$ are given by
\beq\label{eq:RTTq}
\begin{split}
\mathscr R(u_1/u_2)T_1^{\pm}(u_1)T_2^{\pm}(u_2)=T_2^{\pm}(u_2)T_1^{\pm}(u_1)\mathscr R(u_1/u_2),\\
\mathscr R(u_1/u_2)T_1^{+}(u_1)T_2^{-}(u_2)=T_2^{- }(u_2)T_1^{+}(u_1)\mathscr R(u_1/u_2),\\
t_{ij}^{(0,-)}=t_{ji}^{(0,+)}=0,\quad \text{ for }1\lle i<j\lle m+n,\\
t_{ii}^{(0,+)}t_{ii}^{(0,-)}=t_{ii}^{(0,-)}t_{ii}^{(0,+)}=1,\quad \text{ for }i\in \bar I.
\end{split}
\eeq

The Gauss decomposition of $\UqglMNh$ gives generating series
\[
e_{ij}^\pm(u)=\sum_{a\gge 0}e_{ij}^{(a,\pm)}u^{\pm a},\quad f_{ji}(u)=\sum_{a\gge 0}f_{ji}^{(a,\pm)}u^{\pm a},\quad d_k^\pm(u)=\sum_{a\gge 0}d_{k,a}^\pm u^{\pm a},
\]
where $1\lle i< j\lle m+n$ and $k\in \bar I$, such that
\beq\label{eq:Gauss-q}
T^\pm(u)=\Big(\sum_{i<j}f_{ji}^\pm(u)\otimes E_{ji}+1\otimes \mathsf{Id}_{\mathscr V}\Big)\Big(\sum_{k}d_{k}^\pm(u)\otimes E_{kk}\Big)\Big(\sum_{i<j}e_{ij}^\pm(u)\otimes E_{ij}+1\otimes \mathsf{Id}_{\mathscr V}\Big).
\eeq
For $i\in I$, define
\[
x_i^+(u)=\sum_{a\in\Z}x_{i,a}^+u^a:=e_{i,i+1}^+(u)-e_{i,i+1}^-(u),
\]
\[
x_i^-(u)=\sum_{a\in\Z}x_{i,a}^-u^a:=f_{i+1,i}^-(u)-f_{i+1,i}^+(u).
\]
The following relations are known from \cite{Zhy97} or \cite[Theorem 3.5]{Zha16},
\begin{subequations}\label{eq:q-relations}
\begin{align}
&[d_j^\star(u),d_k^\pm(w)]=0,\label{eq:dd-q}\\
&[d_j^\star(u),x_i^\pm(w)]=[x_i^\pm(u),x_l^\pm(w)]=0,\quad \text{ for }|i-j|\gge 2,\ |i-l|\gge 2, \label{eq:dx-q>1}\\
&d_i^\star(u)x_i^\pm(w)=\Big(\frac{q_iu-q_i^{-1}w}{u-w}\Big)^{\mp 1}x_i^\pm(w)d_i^\star(u),\label{eq:dx1}\\
&d_{i+1}^\star(u)x_i^\pm(w)=\Big(\frac{q_{i+1}^{-1}u-q_{i+1}w}{u-w}\Big)^{\mp 1}x_i^\pm(w)d_{i+1}^\star(u),\label{eq:dx2}\\
&[x_i^+(u),x_l^-(w)]=\delta_{il}(q_i-q_i^{-1})\delta(u/w)(d_{i+1}^+(u)d_i^+(u)^{-1}-d_{i+1}^-(w)d_i^-(w)^{-1}),\label{eq:xx}
\end{align}
\end{subequations}
where $i,l\in I$, $j,k\in \bar I$,  $\delta(u)=\sum_{n\in\Z}u^n\in \C[[u^{\pm 1}]]$, and $\star$ is either $+$ or $-$.

\begin{lem}\label{lem:zero-mode0}
We have 
\[
x_{i,0}^+=\big(t_{ii}^{(0,+)}\big)^{-1}t_{i,i+1}^{(0,+)},\qquad x_{i,0}^-=t_{i+1,i}^{(0,-)}\big(t_{ii}^{(0,-)}\big)^{-1}.
\]
\end{lem}
\begin{proof}
Since $t_{ij}^{(0,-)}=t_{ji}^{(0,+)}=0$ for $1\lle i<j\lle m+n$, it is not hard to show from Gauss decomposition \eqref{eq:Gauss-q} that $f_{ji}^+(u)\in u \UqglMNh[[u]]$ and $e_{ij}^-(u)\in u^{-1} \UqglMNh[[u^{-1}]]$. The statement follows that by comparing the constant terms in \eqref{eq:Gauss-q}.
\end{proof}

\begin{lem}[\cite{DF93}]\label{lem:loop-gen}
The quantum affine superalgebra $\UqglMNh$ is generated by the coefficients of $x_i^\pm(u)$ and $d_j^\pm(u)$ for $i\in I$ and $j\in \bar I$.	\qed
\end{lem}

Recall $\gamma_i$ defined in Section \ref{sec:Gauss}.

\begin{lem}
The coefficients of the series $\prod_{j\in\bar I} \big(d^+_{j}(uq^{2\gamma_j})\big)^{s_j}$ are central in $\UqglMNh$.
\end{lem}
\begin{proof}
The lemma follows from Lemma \ref{lem:loop-gen} and relations \eqref{eq:q-relations}.
\end{proof}

Define $\tl t_{ij}^{(a,\pm)}\in \UqglMNh$ for $i,j\in\bar I$ and $a\in \Z_{\gge 0}$ by
\[
\big(T^\pm (u)\big)^{-1}=\sum_{i,j\in\bar I}\sum_{a\gge 0}\tl t_{ij}^{(a,\pm)}u^{\pm a}\otimes E_{ij}\in \UqglMNh\otimes \End(\mathscr V)[[u^{\pm 1}]].
\]
Clearly, $\tl t_{ij}^{(a,\pm)}$ generate $\UqglMNh$.
\begin{lem}\label{lem:q-homo}
	The assignment 
	$$
	\tl t_{ij}^{(a,\pm)}\mapsto \tl t_{i'j'}^{(a,\pm)},\quad i,j\in\bar I,~a\in\Z_{\gge 0},
	$$ extends uniquely to a superalgebra morphism $\phi_r:\UqglMNh\to \mathrm U_q(\widehat\gl_{m'|n})$. Moreover, we have
	\[
	\phi_r: x_i^{\pm}(u)\mapsto x_{i'}^\pm(u),\qquad d_j^\pm(u)\mapsto d_{j'}^\pm(u),\quad \text{ for } ~i\in I,~j\in \bar I.
	\]
\end{lem}
\begin{proof}
	The statement follows from the same strategy of \cite[(5.39-5.41)]{DF93}. More precisely, the first statement follows from the defining relations \eqref{eq:RTTq} for $\mathrm U_q(\gl_{m'|n})$ by taking inverse to $T_1^\pm(u)$, $T_2^\pm(u)$ and restricting to $E_{i'j'}\otimes E_{k'l'}$ for $i,j,k,l\in\bar I$. The second statement then follows from the Gauss decomposition.
\end{proof}

\subsection{Evaluation morphisms and $\ell$-weights}
The quantum superalgebra $\UqglMN$ has another presentation as follows, see \cite[Section 3.1.2]{Zha16}. Let $\mathscr R=\mathscr R(1)$, see \eqref{eq:q-R-matrix}. 

Let $\mathscr U_q(\glMN)$ be the unital associative superalgebra generated by $t_{ij}^\pm$, for $1\lle i<j\lle m+n$, with parity $|i|+|j|$ and with the relations in $\mathscr U_q(\glMN)\otimes \End(\mathscr V^{\otimes 2})$
\[
\mathscr RT_1^\pm T_2^\pm =T_2^\pm T_1^\pm \mathscr R,\quad  \mathscr RT_1^+ T_2^- =T_2^- T_1^+ \mathscr R,\quad t^+_{ii}t^-_{ii}=t^-_{ii}t^+_{ii}=1,
\] 
where $T^+ = \sum_{i\lle j}t^+_{ij}\otimes E_{ij}$ and $T^- = \sum_{i\lle j}t^-_{ji}\otimes E_{ij}$ $\in \mathscr U_q(\glMN)\otimes \End(\mathscr V)$. The superalgebras $\UqglMN$ and $\mathscr U_q(\glMN)$ are isomorphic via the isomorphism extended by the assignments
\[
e_i^+\mapsto \frac{(t_{ii}^+)^{-1}t_{i,i+1}^+}{1-q_i^{-2}},\quad e_i^-\mapsto \frac{t_{i+1,i}^-(t_{ii}^-)^{-1}}{1-q_i^{2}},\quad t_j^{s_j}\mapsto t_{jj}^+=(t_{jj}^-)^{-1}.
\]

The assignment 
\[
t_{ij}^\pm (u)\mapsto \frac{t_{ij}^\pm-u^{\pm 1}t_{ij}^\mp}{1-u^{\pm 1}}
\]
uniquely extends to a superalgebra homomorphism $\varpi_{m|n}:\UqglMNh\to \mathscr U_q(\glMN)\cong \UqglMN$. We call $\varpi_{m|n}$ the \emph{evaluation morphism}. Let $M$ be a $\UqglMN$-module. We call the module obtained by pulling back $M$ through the evaluation morphism an \emph{evaluation module} and denote it again by $M$.

The assignment $t_{ij}^\pm\mapsto t_{ij}^{(0,\pm)}$ defines a superalgebra morphism $\iota:\mathscr U_q(\glMN)\to \UqglMNh$. Moreover, we clearly have $\varpi_{m|n}\circ \iota =\mathsf{Id}_{\mathscr U_q(\glMN)}$. Hence $\mathscr U_q(\glMN)\cong \UqglMN$ is identified as a subalgebra of $\UqglMNh$ which is invariant under the evaluation morphism.

\begin{lem}\label{lem:zero-mode}
Identifying	$\UqglMN$ as a subalgebra of $\UqglMNh$, we have $x_{i,0}^\pm=(1-q_i^{\mp 2})e_{i}^\pm$.
\end{lem}
\begin{proof}
The lemma follows from Lemma \ref{lem:zero-mode0} and the observations above.	
\end{proof}

Set $\mathcal B:=\C[[u^{-1}]]^\times$ and $\mathfrak B:=\mathcal B^{\bar I}$. We call an element $\bm\zeta\in \mathfrak B$ an \emph{$\ell$-weight}. We write $\ell$-weights in the form $\bm\zeta=(\zeta_i(u))_{i\in \bar I}$, where $\zeta_i(u)\in \mathcal B$ for all $i\in \bar I$. 

Let $M$ be a $\UqglMNh$-module. We say that a nonzero vector $v\in M$ is \emph{of $\ell$-weight $\bm \zeta$} if $d_{i}^+(u)v=\zeta_i(u)v$ for $i\in \bar I$. We say that a vector $v\in M$ is a \emph{highest $\ell$-weight vector of $\ell$-weight $\bs\zeta$} if $v$ is of $\ell$-weight $\bs\zeta$ and $x_i^+(u)v=0$ for all $i\in I$. 

Let $M$ be a finite-dimensional $\UqglMNh$-module and $\bm\zeta\in \mathfrak B$ an $\ell$-weight. Let
\[
\zeta_i(u)=\sum_{j\gge 0} \zeta_{i,j}u^{-j},\qquad \zeta_{i,0}\in \C^\times,~\zeta_{i,j}\in \C.
\]
Denote by $M_{\bm \zeta}$ the \emph{generalized $\ell$-weight space} corresponding to the $\ell$-weight $\bm\zeta$,
\[
M_{\bm\zeta}:=\{v\in M~|~  (d_{i,j}^+-\zeta_{i,j})^{\dim M} v=0 \text{ for all }i\in \bar I, \ j\in \Z_{\gge 0}\}.
\]
We call $M$ \emph{thin} if $\dim(M_{\bm\zeta})\lle 1$ for all $\bm\zeta\in\mathfrak B$. In particular, if $M$ is thin, then the actions of $d_i^+(u)$ are simultaneously diagonalizable.

\begin{eg}
Let $\la$ be a $\glMN$-weight. Then the evaluation module $\mathscr L(\la)$ has the highest $\ell$-weight $\bm \zeta$,
\[
\zeta_i(u)=\frac{q^{\la_i}-uq^{-\la_i}}{1-u},\quad \zeta_j(u)=\frac{q^{-\la_j}-uq^{\la_j}}{1-u},\quad 1\lle i\lle m,\quad m+1\lle j\lle m+n.\qedd
\] 
\end{eg}

\subsection{Skew representations}
Recall that we have a fixed $r\in\Z_{\gge 0}$ and $k':=r+k$ for all $k\in\Z$. We identify $\mathrm{U}_q(\gl_r)$ as a subalgebra of $\mathrm{U}_q(\gl_{m'|n})$ via the natural embedding $e^\pm_{i}\mapsto e^\pm_i$ and $t_j^{\pm 1}\mapsto t_j^{\pm 1}$. The quantum superalgebra $\UqglMN$ is identified as a subalgebra of $\mathrm{U}_q(\gl_{m'|n})$ via the embedding $e_i^\pm\to e_{i'}^\pm$ and $t_j^{\pm 1}\mapsto t_{j'}^{\pm 1}$. Clearly, the two subalgebras $\mathrm{U}_q(\gl_r)$ and $\UqglMN$ commute with each other.

Let $\la=(\la_1,\dots,\la_{m'+n})$ be an $(m'|n)$-covariant weight and $\mu=(\mu_1,\dots,\mu_r)$ a $(r|0)$-covariant weight. Let $\mathscr L(\la/\mu)$ be the subspace of $\mathscr L(\la)$ given by
\[
\mathscr L(\la/\mu):=\{v\in\mathscr L(\la)~|~t_iv=q^{\mu_i}v, e_j^+v=0,\text{ for }1\lle i\lle r, 1\lle j<r \}.
\]
It is clear that $\mathscr L(\la/\mu)$ is a $\UqglMN$-module and a $\mathrm U_q(\gl_{m'|n})^{\mathrm{U}_q(\gl_r)}$-module. It follows from Theorem \ref{thm:q-GT-formulas} that $\mathscr L(\la/\mu)$ admits a basis $\xi_{\La}$ parameterized by all $\la/\mu$-admissible GT tableaux $\La$. Moreover,
\beq\label{eq:matrix-q}
e_k^\pm\xi_{\La}=\sum_{i=1}^{k'}[\mathscr E^\pm_{\La,ki}]\ \xi_{\La\pm\delta_{ki}},
\eeq
where $[\mathscr E^\pm_{\La,ki}]$ are obtained from $\mathscr E^\pm_{\La,ki}$ by replacing all numbers appearing in $\mathscr E^\pm_{\La,ki}$ from Theorem \ref{thm:gt-formula} to their associated $q$-numbers.

The quantum superalgebra $\mathrm U_q(\widehat\gl_r)$ is identified as a subalgebra of $\mathrm U_q(\widehat\gl_{m'|n})$ via the natural embedding $t_{ij}^{(a,\pm)}\to t_{ij}^{(a,\pm)}$. Recall $\phi_r$ from Lemma \ref{lem:q-homo}.

\begin{lem}\label{lem:q-supercommute}
The image of $\UqglMNh$ under $\phi_r$ in $\mathrm U_q(\widehat\gl_{m'|n})$ supercommutes with the subalgebra $\mathrm U_q(\widehat\gl_r)$.
\end{lem}
\begin{proof}
The subalgebra $\mathrm U_q(\widehat\gl_r)$ of $\mathrm U_q(\widehat\gl_{m'|n})$ is generated by coefficients of $x_i^{\pm}(u)$ and $d_j^{\pm}(u)$ for $1\lle i\lle r-1$ and $1\lle j\lle r$. The image of $\UqglMNh$ under $\phi_r$ is generated by coefficients of $x_{i'}^{\pm}(u)$ and $d_{j'}^{\pm}(u)$ for $i\in I$ and $j\in \bar I$. Now the lemma follows from \eqref{eq:q-relations}.
\end{proof}

\begin{rem}
Using Lemma \ref{lem:q-homo}, the lemma reduces to show that $t_{ij}^{\pm}(u)$ supercommutes with $t_{kl}^{\pm }(u)$ for $1\lle i,j\lle r$ and $1'\lle k,l\lle m'+n$ in $\mathrm U_q(\widehat\gl_{m'|n})$. This can be shown directly from \eqref{eq:RTTq}, see e.g. \cite[equation (5)]{Gow07}.\qed
\end{rem}

It follows from Lemma \ref{lem:q-supercommute} that the image of the homomorphism
\[
\varpi_{m'|n}\circ \phi_r:\UqglMNh \to \mathrm U_q(\gl_{m'|n})
\]
supercommutes with the subalgebra $\mathrm U_q(\gl_r)$ in $\mathrm U_q(\gl_{m'|n})$. This implies that the subspace $\mathscr L(\la/\mu)$ is invariant under the action of the image of $\varpi_{m'|n}\circ \phi_r$. Therefore, $\mathscr L(\la/\mu)$ is a $\UqglMNh$-module. We call $\mathscr L(\la/\mu)$ a \emph{skew representaton} of $\UqglMNh$.

\begin{rem}
Note that skew representations $\mathscr L(\la/\mu)$ can also be defined using a Lie superalgebra $\gl_{r_1|r_2}$ instead of a Lie algebra $\gl_r$, see \cite[Section 3]{LM20}. However, the associated skew representation essentially depends on the shape of the skew Young diagram $\Gamma_{\la/\mu}$, see \cite[Remark 3.7]{LM20}. Here we only treat the case of $\gl_r$ for simplicity.\qed
\end{rem}

Define the series $\mathcal C_k(u)$, for $0\lle k\lle m+n$, in $\mathrm U_q(\widehat\gl_{m'|n})$,
\[
\mathcal C_k(u):=\prod_{i=1}^r d_i^+(uq^{-2(r-i+1)})\prod_{j=1}^k\big(d_{j'}^+(uq^{2\gamma_j})\big)^{s_j}
\] 
For a $\la/\mu$-admissible GT tableau $\La$, define rational functions $\mathscr Y_{\La,k}$, for $0\lle k\lle m+n$, by
\[
\mathscr Y_{\La,k}(u)=\prod_{i=1}^r\Big(q^{\la_{k'i}}\frac{1-uq^{-2l_{ki}}}{1-uq^{-2(r-i+1)}}\Big)\prod_{j=1}^k\Big(q_j^{\la_{k'j'}}\frac{1-uq^{-2l_{kj'}}}{1-uq^{2\gamma_j}}\Big)^{s_j}.
\]
Define $\zeta_{\La,k}(u)$ by
\beq\label{eq:l-weight-q}
\zeta_{\La,k}(u)=\left(\frac{\mathscr Y_{\La,k}(uq^{-2\gamma_{k}})}{\mathscr Y_{\La,k-1}(uq^{-2\gamma_{k}})}\right)^{s_k}.
\eeq
and set $\bm\zeta_{\La}=(\zeta_{\La,1}(u),\dots,\zeta_{\La,m+n}(u))$.

\begin{lem}\label{lem:main-q}
We have $\mathcal C_k(u)\xi_{\La}=\mathscr Y_{\La,k}(u)\xi_{\La}$ for $0\lle k\lle m+n$. Moreover, the vector $\xi_{\La}\in \mathscr L(\la/\mu)$ is of $\ell$-weight $\bm\zeta_{\La}$. In particular, the skew representation $\mathscr L(\la/\mu)$ is thin. 
\end{lem}
\begin{proof}
The proofs of statements are similar to those of Lemmas \ref{lem:A-action},  \ref{lem:l-weight}, and Corollary \ref{cor:thin}. 	
\end{proof}

For each $i\in I$ and $a\in\C$, define the \emph{simple $\ell$-root} $A_{i,a}\in \mathfrak B$ by
\[
(A_{i,a})_j(u)=\frac{u-a}{q^{(\alpha_i,\epsilon_j)}u-q^{-(\alpha_i,\epsilon_j)}a},\qquad j\in \bar I.
\]
The following proposition will be used in the proof of one of our main results.

\begin{prop}[{\cite[Proposition 3.1]{You15}}]\label{prop:main-q}
Let $V$	be a finite-dimensional $\UqglMNh$-module. Pick and fix any $i\in I$. Let $(\bm \mu,\bm \nu)$ be a pair of $\ell$-weights of $V$ such that $x_{i,j}^{\pm}(V_{\bm\mu})\cap V_{\bm\nu}\ne \{0\}$ for some $j\gge 1$. Then:
\begin{enumerate}
\item $\bm \nu=\bm\mu A_{i,a}^{\pm 1}$ for some $a\in \C^\times $,
\item there exist bases $(v_k)_{1\lle k\lle \dim(V_{\bm\mu})}$ of $V_{\bm\mu}$ and $(w_l)_{1\lle l\lle \dim(V_{\bm\nu})}$ of $V_{\bm\nu}$, and complex polynomials $P_{k,l}(z)$ of degree $\lle k+l -2$ such that
\[
(x_i^{\pm}(z)v_k)_{\bm \nu}=\sum_{l =1}^{\dim(V_{\bm \nu})} w_l P_{k,l}(\pa_a)\delta(a/z).\qedd
\]
\end{enumerate}
\end{prop}

\subsection{Main results}
Now we are ready to formulate and prove the main results of this section.
\begin{thm}\label{thm:q-main}
We have 
$$
d_k^+(u)\xi_{\La}=\zeta_{\La,k}(u)\xi_{\La},
$$
$$
x_k^{\pm}(u)\xi_{\La}=(1-q_i^{\mp 2})\sum_{i=1}^{k'}[\mathscr E_{\La,ki}^{\pm}]\delta(q^{2(l_{ki}+s_{ki}^\pm+\gamma_k)}/u)\xi_{\La\pm\delta_{ki}},
$$
where $l_{ki}$, $s_{ki}^\pm$, $[\mathscr E_{\La,ki}^{\pm}]$ and $\zeta_{\La,k}(u)$ are defined in \eqref{eq:lki}, \eqref{eq:ski}, \eqref{eq:matrix-q} and \eqref{eq:l-weight-q}, respectively.
\end{thm}
\begin{proof}
The proof is similar to that of Theorem \ref{thm:main}.	We briefly describe the strategy of proof. By Lemma \ref{lem:main-q} and Proposition \ref{prop:main-q}, it suffices to find the number $a$ appearing in Proposition \ref{prop:main-q} for each pair $\La$ and $\La\pm\delta_{ki}$ and the action of $x_{i,0}^\pm$ on $\xi_{\La}$. The number $a$ is computed by Lemma \ref{lem:main-q} while the action of $x_{i,0}^\pm$ on $\xi_{\La}$ is obtained from Theorem \ref{thm:q-GT-formulas} and Lemma \ref{lem:zero-mode}.
\end{proof}

\begin{thm}\label{thm:q-irr}
Every skew representation of $\UqglMNh$ is irreducible.	
\end{thm}
\begin{proof}
The proof is similar to that of Theorem \ref{thm:irr} using Theorem \ref{thm:q-main}.
\end{proof}

An analogue of Theorem \ref{thm:thin-one} holds for $\mathrm U_q(\widehat{\gl}_{1|1})$ as well which we do not repeat.

\end{document}